\documentclass[11pt, a4paper]{amsart}
\usepackage{amssymb, array, amsmath, amscd, pdfpages, enumerate, amsthm, fixltx2e, setspace, hyperref,mathtools}

\usepackage[utf8]{inputenc}
\usepackage{amssymb}
\usepackage{bm}
\usepackage{graphicx} 
\usepackage{ifthen}

\newcommand{\Dom}{D}
\newcommand{\eps}{\varepsilon}

\DeclareMathOperator{\re}{Re}

\DeclareMathOperator{\diag}{diag}

\newcommand*{\C}{{\mathbb{C}}}     
\newcommand*{\R}{{\mathbb{R}}}     
\newcommand*{\Z}{{\mathbb{Z}}}     
     
\newcommand*{\N}{{\mathbb{N}}}

\newcommand*{\Lin}{{\mathcal{L}}}   
\newcommand{\ran}{{\mathcal{R}}}   
\renewcommand{\ker}{{\mathcal{N}}}

\newcommand*{\abs}[1]{\lvert#1\rvert}
\newcommand*{\norm}[1]{\lVert#1\rVert}
\newcommand*{\set}[1]{\{#1\}}
\newcommand*{\setm}[2]{\{\,#1\mid#2\,\}}   
\newcommand*{\iprod}[2]{\langle#1,#2\rangle}    

\newcommand*{\ldelim}[2]{\csname#1l\endcsname#2} 
\newcommand*{\rdelim}[2]{\csname#1r\endcsname#2} 
\newcommand*{\mdelim}[2]{\csname#1m\endcsname#2} 
\newcommand*{\Set}[2][default]{\ifthenelse{\equal{#1}{default}}{\left\{#2\right\}}{\ldelim{#1}{\{}#2\rdelim{#1}{\}}}} 
\newcommand{\Setm}[3][big]{\ldelim{#1}{\{}\,#2\mdelim{#1}{|}#3\,\rdelim{#1}{\}}} 
\newcommand*{\Lp}[1][p]{L^{#1}}

\newcommand*{\lp}[1][p]{\ell^{#1}} 

  \newcommand{\pmat}[1]{\begin{bmatrix}#1\end{bmatrix}}
\newcommand{\pmatsmall}[1]{\begin{bsmallmatrix}#1\end{bsmallmatrix}}

\newcommand{\from}{\leftarrow}
\newcommand*{\List}[2][1]{\set{#1,\ldots,#2}}

\newcommand{\eq}[1]{\begin{align*}#1\end{align*}}
\newcommand{\eqn}[1]{\begin{align}#1\end{align}}
\newcommand{\ieq}[1]{$#1$}

\newcommand{\gs}{\sigma}
\newcommand{\ga}{\alpha}
\newcommand{\gb}{\beta}
\renewcommand{\gg}{\gamma}
\newcommand{\gd}{\delta}
\newcommand{\gl}{\lambda}
\newcommand{\gw}{\omega}

\newcommand{\inv}{^{-1}}
\newcommand*{\pinv}{^{\dagger}} 

\newcommand*{\ddb}[2][1]{\ifthenelse{\equal{#1}{1}}{\frac{d}{d#2}}{\frac{d^{#1}}{d#2^{#1}}}}
\newcommand*{\pd}[3][1]{\ifthenelse{\equal{#1}{1}}{\frac{\partial{#2}}{\partial{#3}}}{\frac{\partial^{#1}{#2}}{\partial#3^{#1}}}}

\newcommand*{\keyterm}[1]{\emph{#1}}

\newcommand{\citel}[2]{\cite[#2]{#1}}

\renewcommand{\pmat}[1]{\begin{bmatrix}#1\end{bmatrix}}
\renewcommand{\pmatsmall}[1]{\begin{bsmallmatrix}#1\end{bsmallmatrix}}

\newcommand{\qext}{q_{ext}}

\newcommand{\CL}{C_\Lambda}

\newcommand{\CeL}{C_{e\Lambda}}

\newcommand{\Pez}{P_{e0}}

\newcommand{\Ac}{A_c}

\newcommand{\Bc}{B_c}
\newcommand{\Bck}[1][k]{B_{c#1}}
\newcommand{\Cc}{C_c}
\newcommand{\CcL}{C_{c\Lambda}}
\newcommand{\Cck}[1][k]{C_{c#1}}
\newcommand{\Dc}{D_c}
\newcommand{\Dcone}{D_{c1}}
\newcommand{\Dctwo}{D_{c2}}
\newcommand{\Dconeadd}{D_{c1}}
\newcommand{\Dtot}{D_{tot}}
\newcommand{\dc}{d_c}

\newcommand{\uc}{u_c}
\newcommand{\yc}{y_c}

\newcommand{\gwlim}{R}

\newcommand{\G}{G}
\newcommand{\Gmu}{G^\mu}

\newcommand{\XB}{X_B}
\newcommand{\XBBd}{X_{B,B_d}}
\newcommand{\ZG}{Z_{\Bc}}

\newcommand{\SIndset}{\mathcal{I}}

\DeclareMathOperator{\blkdiag}{blockdiag}

\newcommand{\PARsysopspert}{(\tilde{A},\tilde{B},\tilde{B}_d,\tilde{C},\tilde{D})}
\newcommand{\PARsys}{(A,B,C,D)}
\newcommand{\PARsysS}{(A^S,B^S,C^S,D^S)}

\newcommand{\PARcontr}{(\Ac,\Bc,\Cc,\Dc)}
\newcommand{\PARcontrone}{(\Ac,\Bc,\Cc,\Dcone)}

\newtheorem{theorem}{Theorem}[section]
\newtheorem{lemma}[theorem]{Lemma}
\newtheorem{proposition}[theorem]{Proposition}
\newtheorem{corollary}[theorem]{Corollary}
\theoremstyle{definition}
\newtheorem{definition}[theorem]{Definition}
\newtheorem{assumption}[theorem]{Assumption}
\newtheorem{remark}[theorem]{Remark}

\numberwithin{equation}{section}

\newenvironment{RORP}{\textbf{The Robust Output Regulation Problem.}\it}{}

\newcommand{\distref}{w_{\mbox{\scriptsize\textit{ext}}}}

\newcommand{\distrefk}{w^k_{\mbox{\scriptsize\textit{ext}}}}

\newcommand{\yrefk}{y^k_{\mbox{\scriptsize\textit{ref}}}}
\newcommand{\wdistk}{w^k_{\mbox{\scriptsize\textit{dist}}}}
\newcommand{\wdisthat}{\hat{w}_{\mbox{\scriptsize\textit{dist}}}}

\newcommand{\Pitrans}{\Pi_{\mbox{\scriptsize\textit{ext}}}}
\newcommand{\Pitransk}[1][k]{\Pi_{\mbox{\scriptsize\textit{ext}}}(#1)}
\newcommand{\Pitranskone}[1][k]{\Pi_{\mbox{\scriptsize\textit{ext}}}^1(#1)}
\newcommand{\Pitransktwo}[1][k]{\Pi_{\mbox{\scriptsize\textit{ext}}}^2(#1)}

\newcommand{\Pitranspert}{\tilde{\Pi}_{\mbox{\scriptsize\textit{ext}}}}
\newcommand{\Pitranskpert}[1][k]{\tilde{\Pi}_{\mbox{\scriptsize\textit{ext}}}(#1)}
\newcommand{\Piseqnorm}{M_{\mbox{\scriptsize\textit{ext}}}}

\newcommand{\Pipartk}{u_k}
\newcommand{\Pipartkalt}{\tilde{u}_k}

\newcommand{\yref}{y_{\mbox{\scriptsize\textit{ref}}}}
\newcommand{\wdist}{w_{\mbox{\scriptsize\textit{dist}}}}

\newcommand{\Mlog}{M_{\textup{log}}}

\begin{document}

\title{Stability and Robust Regulation of Passive Linear Systems}

\thispagestyle{plain}

\author{Lassi Paunonen}
\address{Mathematics and Statistics, Faculty of Information Technology and Communication Sciences, Tampere University, PO.\ Box 692, 33101 Tampere, Finland.}
\thanks{
  The manuscript was completed while the author was visiting Professor Charles J.K.\ Batty at University of Oxford from January to June in 2017. 
The research is funded by the Academy of Finland grant numbers 298182 and 310489.} 
\email{lassi.paunonen@tuni.fi}

\begin{abstract}
We study the stability of coupled impedance passive regular linear systems under power-preserving interconnections. We present new conditions for strong, exponential, and non-uniform stability of the closed-loop system. We apply the stability results to the construction of passive error feedback controllers for robust output tracking and disturbance rejection for strongly stabilizable passive systems. In the case of nonsmooth reference and disturbance signals we present conditions for non-uniform rational and logarithmic rates of convergence of the output. The results are illustrated with examples on designing controllers for linear wave and heat equations, and on studying the stability of a system of coupled partial differential equations.
\end{abstract}


\maketitle

\section{Introduction}
\label{sec:intro}

In this paper we study the stability properties and control of regular linear systems~\cite{Wei94} of the form\footnote{Here $\CL$ and $\CcL$ denote the $\Lambda$-extensions of $C$ and $\Cc$, respectively. See Section~\ref{sec:CLsysdefs} for details.}
\begin{subequations}
  \label{eq:plantintro}
  \eqn{
  \dot{x}(t)&=Ax(t)+Bu(t) ,
  \qquad x(0)=x_0\in X,\\
  y(t)&=\CL x(t)+Du(t)
  } 
\end{subequations}
on a Hilbert space $X$, where $u(t)$ is the input of the system and $y(t)$ is the output.  
Our main interest is in systems that are  \keyterm{impedance passive}~\cite{BroLoz07book,Sta02,TucWei14} (or \keyterm{passive} for short) in the sense that their solutions satisfy
\eq{
  \ddb{t}\norm{x(t)}^2 \leq 2\re \iprod{u(t)}{y(t)}, \qquad t> 0.
}
Passive systems are encountered especially in the study of mechanical or electrical systems modeled with partial differential equations.  In particular,~\eqref{eq:plantintro} is impedance passive if $A$ generates a contraction semigroup, $B$ and $C$ are bounded operators, $C=B^\ast$, and $\re D\geq 0$.

The paper consists of two main parts.
In the first part we focus on the stability of the coupled system consisting of~\eqref{eq:plantintro} and another passive regular linear system
\begin{subequations}
  \label{eq:sys2intro}
  \eqn{
  \dot{z}(t)&=\Ac z(t)+\Bc u_c(t),  \qquad z(0)=z_0\in Z,\\
  y_c(t)&=\CcL z(t)+\Dc u_c(t)
  } 
\end{subequations}
with $\Dc^\ast=\Dc$
under a \keyterm{power-preserving interconnection} where
\eq{
  u(t)=y_c(t), \qquad u_c(t)=-y(t).
}
We study the stability of the resulting closed-loop system
\eqn{
  \label{eq:CLsysintro}
  \dot{x}_e(t)&=A_ex_e(t), \qquad x_e(0)=x_{e0}\in X_e
}
on the Hilbert space $X_e=X\times Z$.
The notation $(\Ac,\Bc,\Cc,\Dc)$ and our results on the closed-loop stability 
are motivated by the second part of the paper where we study 
robust output tracking and disturbance rejection for the system~\eqref{eq:plantintro}. 
In this situation~\eqref{eq:sys2intro} is an unstable dynamic feedback controller.
However, our results are also applicable
when the roles of the systems are reversed, i.e., when~\eqref{eq:sys2intro} is a system to be controlled and~\eqref{eq:plantintro} is the controller, 
and they can also be used to 
study the stability of systems of partial differential equations coupled on the boundary or inside the domain.
Our main interest is in the situation where $\Ac$ has a countable number of spectral points on the imaginary axis.

We study~\eqref{eq:CLsysintro} 
in terms of the stability properties of the strongly continuous semigroup $T_e(t)$ generated by $A_e: \Dom(A_e)\subset X_e\to X_e$.
As our main results we introduce conditions under which the semigroup $T_e(t)$ is exponentially stable, strongly stable, or \keyterm{non-uniformly stable}~\cite{BatDuy08,RozSei19}.
Among these, exponential stability is the strongest form of stability.
However, in certain control applications exponential stability is unachievable, and many 
partial differential equations and coupled systems are known to lack exponential decay of energy.  
These situations arise especially in wave equations with partial damping and in 
coupled hyperbolic-parabolic systems~\cite{ZhaZua04, AvaLas16}.
Recently many such coupled systems 
have been shown to be \keyterm{polynomially stable}~\cite{LiuRao05,BatDuy08,BorTom10}, which means that the classical solutions of the system decay at rational rates, i.e., for some constants $M_e,\ga,t_0>0$
\eq{
  \norm{T_e(t)x_{e0}} \leq \frac{M_e}{t^{1/\ga}}\norm{A_ex_{e0}}, \qquad x_{e0}\in \Dom(A_e), \ t\geq t_0 .
}
In this paper we introduce new results 
for studying polynomial and the more general non-uniform stability
for coupled passive abstract linear systems~\eqref{eq:plantintro} and~\eqref{eq:sys2intro}.

Strong and exponential closed-loop stabilities of infinite-dimensional systems have been studied in the literature for passive one-dimensional boundary control systems~\cite{VilZwa09,RamLeG14}, coupled systems with collocated inputs and outputs~\cite{FenGuo15}, and passive systems coupled with finite-dimensional systems~\cite{ZhaWei17}.
Polynomial stability of coupled systems has been studied extensively in the context of coupled linear partial differential equations~\cite{AmmNic09,FatLea12,AvaLas16,AmmDim16}, and for abstract hyperbolic-parabolic systems~\cite{BenAmm16}.

In the second part of the paper we study the \keyterm{robust output regulation problem} where the aim is to design a controller in such a way that the output $y(t)$ of the system~\eqref{eq:plantintro} converges to a given reference signal $\yref(t)$ asymptotically in the sense that
\eq{
  \norm{y(t)-\yref(t)}  \to 0, \qquad t\to \infty
}
despite possible external disturbance signals $\wdist(t)$.
In addition, the controller is required to be \keyterm{robust} in the sense that it should achieve output tracking even if the parameters $\PARsys$ experience small changes or contain uncertainties.
This control problem
has been studied actively in the literature for various classes of infinite-dimensional linear systems~\cite{YamHar88,LogTow97,HamPoh00,RebWei03,Imm07a,HamPoh10,PauPoh10,WanJi14b} including regular linear systems~\cite{WeiHaf99, BouHad09,PauPoh14a,XuSal14,Pau16a,Pau17b} and passive systems~\cite{RebWei03}. 

The robust output regulation problem can be solved with a dynamical error feedback controller
of the form
\begin{subequations}
  \label{eq:contrintro}
  \eqn{
  \dot{z}(t)&=\Ac z(t)+\Bc(\yref(t)-y(t)), \qquad z(0)=z_0\in Z,\\
  u(t)&=\CcL z(t) + \Dc (\yref(t)-y(t)).
  }
\end{subequations} 
One of the fundamental results of the theory, the \keyterm{internal model principle}~\cite{FraWon75a,Dav76,PauPoh10,PauPoh14a}, 
implies that robust output tracking 
can be achieved by
including a suitable 
number of copies
of the frequencies $\set{\gw_k}_{k\in\SIndset}$
of $\yref(t)$ and $\wdist(t)$
into the dynamics of the controller and using the remaining parameters of~\eqref{eq:contrintro} to stabilize the closed-loop system.
While the inclusion of the \keyterm{internal model} is both necessary and sufficient for robustness, 
the resulting closed-loop can be stabilized in various ways.
Under fairly general assumptions the closed-loop stability can be achieved with observer-based design methods~\cite{HamPoh10,Pau16a} leading to infinite-dimensional controllers. 
If the system~\eqref{eq:plantintro} 
can be stabilized exponentially with output feedback
and if $\yref(t)$ and $\wdist(t)$ 
contain a finite number of frequencies,
then 
 $\Ac$ can be chosen to
be \keyterm{minimal} in the sense that it contains only
 the internal model, and the closed-loop system can be stabilized with suitable choices of $\Bc$ and $\Cc$~\cite{LogTow97,HamPoh00,RebWei03}. 
 It was shown in~\citel{RebWei03}{Thm. 1.2} that if~\eqref{eq:plantintro} is passive and exponentially stabilizable, then robust output regulation can be achieved in a natural way using a minimal passive controller~\eqref{eq:contrintro}.

In this paper we extend the passive controller design presented in~\cite{RebWei03}.
We present a robust passive controller for systems~\eqref{eq:plantintro}
that are not exponentially stablizable, but only strongly stabilizable. Such systems are encountered, for example, in control of wave equations, as illustrated in Section~\ref{sec:examples}. Moreover, our design methods allow considering nonsmooth periodic reference and disturbance signals with infinite numbers of frequencies.
In earlier references, the robust output regulation of nonsmooth signals has only been achieved using an observer in the controller~\cite{HamPoh10,Pau17b}. We solve this problem with two new robust controllers having the property that $\Ac$ contains only the internal model of the reference and disturbance signals. These controllers achieve either exponential, polynomial, or non-uniform closed-loop stability depending on the properties of the system~\eqref{eq:plantintro} and the choices of the controller's parameters.
In the case of non-uniform closed-loop stability we present non-uniform rates of convergence for the output $y(t)$ for sufficiently smooth $\yref(\cdot)$ and $\wdist(\cdot)$.

One of the passive controllers presented in this paper is based on a transport equation with boundary control and observation, and under suitable assumptions on the system~\eqref{eq:plantintro} (in general requiring $D\neq 0$) the controller achieves robust output regulation of all $\tau$-periodic reference and disturbance signals with exponential convergence rate of the output.
This structure is related to the controllers used in repetitive control~\cite{HarYam88,WeiHaf99} and in~\cite{Imm07a}.

The paper is organised as follows. In Section~\ref{sec:CLsysdefs} we state the main standing assumptions. The results on stability of the closed-loop system are presented in Section~\ref{sec:CLstab}. In Section~\ref{sec:RORP} we formulate the robust output regulation problem, and the results on construction of robust controllers are presented in Section~\ref{sec:ContrConstructions}. In Section~\ref{sec:examples} we illustrate the controller construction for concrete partial differential equations, including two one-dimensional wave equations and a two-dimensional heat equation. Appendix~\ref{sec:OpEstimates} collects helpful lemmata that are used throughout the paper.

\section{Notation and Definitions}
\label{sec:CLsysdefs}

If $X$ and $Y$ are Banach spaces and $A:X\rightarrow Y$ is a linear operator, we denote by $\Dom(A)$, $\ker(A)$ and $\ran(A)$ the domain, kernel and range of $A$, respectively. The space of bounded linear operators from $X$ to $Y$ is denoted by $\Lin(X,Y)$. If \mbox{$A:X\rightarrow X$,} then $\gs(A)$
and $\rho(A)$ denote the spectrum
and the \mbox{resolvent} set of $A$, respectively. For $\gl\in\rho(A)$ the resolvent operator is  \mbox{$R(\gl,A)=(\gl -A)^{-1}$}.  The inner product on a Hilbert space is denoted by $\iprod{\cdot}{\cdot}$.
For $T\in \Lin(X)$ on a Hilbert space $X$ we define $\re T = \frac{1}{2}(T+T^\ast)$.
The Moore-Penrose pseudoinverse of $T\in \Lin(X,Y)$ is denoted by $T\pinv$.
For two functions $f:I\subset \R\to X$ and $g:\R_+\to \R_+$ we write $\norm{f(t)}=O(g(\abs{t}))$ if there exist $M_g,T_g>0$ such that $\norm{f(t)}\leq M_g g(\abs{t})$ whenever $\abs{t}\geq T_g$.
We denote $f(t)\lesssim g(t)$ and $f_k\lesssim g_k$ if there exist $M_1,M_2>0$ such that $f(t)\leq M_1 g(t)$ and $f_k\leq M_2g_k$ for all values of the parameters $t$ and $k$.

  In Sections~\ref{sec:RORP} and~\ref{sec:ContrConstructions} we also consider the system~\eqref{eq:plantintro} 
on a Hilbert space $X$ 
  with an additional disturbance signal input $\wdist(t)$, i.e.,
\begin{subequations}
  \label{eq:plantfull}
  \eqn{
  \dot{x}(t)&=Ax(t)+Bu(t) + B_d\wdist(t), \qquad x(0)=x_0\in X,\\
  y(t)&=\CL x(t)+Du(t).
  } 
\end{subequations}
Throughout the paper the operators $B\in \Lin(U,X_{-1})$, $B_d\in \Lin(U_d,X_{-1})$ and $C\in \Lin(X_1,Y)$ are admissible~\citel{TucWei09book}{Sec. 4} with respect to the semigroup $T(t)$ generated by $A: \Dom(A)\subset X\to X$.
Here $U$, $U_d$, and $Y$ are Hilbert spaces,
the space $X_1=\Dom(A)$ is equipped with the graph norm of $A$, and $X_{-1}$ is the completion of $X$ with respect to the norm $\norm{x}_{-1}=\norm{R(\gl_0,A)x}$ where $\gl_0\in \rho(A)$ is arbitrary and fixed.
We assume that the system $(A,[B,B_d],\CL,D)$ in~\eqref{eq:plantfull} with input $(u(t),\wdist(t))\in U\times U_d$ and output $y(t)\in Y$ is a regular linear system~\citel{Wei94}{Sec.~5}.
We denote $\XB=\Dom(A)+\ran(R(\gl_0,A)B)$ and $\XBBd=\Dom(A)+\ran(R(\gl_0,A)[B,B_d])$.
The $\Lambda$-extension of $C$ is
\ieq{\CL x=\lim_{\gl\to\infty} \gl CR(\gl,A)x},
where $\Dom(\CL)$ consists of those $x\in X$ for which the limit exists.
The regularity of~\eqref{eq:plantfull} implies that $\ran(R(\gl,A)B)\subset \Dom(\CL)$ and $\ran(R(\gl,A)B_d)\subset \Dom(\CL)$ for all $\gl\in\rho(A)$ and that the transfer functions $P(\cdot): \hat{u}\mapsto \hat{y}$ and $P_d(\cdot): \wdisthat\mapsto \hat{y}$ have the formulas
\eq{
P(\gl) = \CL R(\gl,A)B+D, \qquad P_d(\gl)=\CL R(\gl,A)B_d.
}
  Throughout the paper we assume that $Y=U$ and that $\PARsys$ is impedance passive~\cite{BroLoz07book,Sta02,TucWei14},
which is equivalent to the property that $\re \iprod{Ax+Bu}{x}\leq \re \iprod{\CL x+Du}{u}$ for all $x\in X$ and $u\in U$ satisfying $Ax+Bu\in X$~\cite[Thm. 4.2]{Sta02}. Under this assumption the semigroup $T(t)$ generated by $A$ is contractive,  
  $\re D\geq 0$, and $\re P(\gl)\geq 0$ for all $\gl\in \C_+$ (such transfer functions are called \keyterm{positive})

We frequently use the following operator identity, see e.g.~\citel{WeiXu05}{Proof of Thm. 1.2}.
  For completeness, we give a proof of the lemma in  Appendix~\ref{sec:OpEstimates}.
 \begin{lemma}
   \label{lem:Woodbury}
   Let $\PARsys$ 
   be a regular linear system and let $Q\in  \Lin(Y,U)$ be invertible. If $\gl\in\rho(A)$ and if
    $Q\inv + \CL R(\gl,A)B$ is boundedly invertible, then $\gl\in\rho(A-BQ\CL)$ and 
   \eq{
     R(\gl,A-BQ\CL) = R(\gl,A)-R(\gl,A)B (Q\inv + \CL R(\gl,A) B)\inv \CL R(\gl,A),
   }
   where $\Dom(A-BQ\CL)=\setm{x\in \Dom(\CL)}{(A-BQ\CL)x\in X}$.
 \end{lemma}

The system~\eqref{eq:sys2intro} is assumed to be another
impedance passive regular linear system on a Hilbert space $Z$ with $\Dc^\ast =\Dc$. 
The scale spaces $Z_1$ and $Z_{-1}$ are defined analogously as $X_1$ and $X_{-1}$. We define $\ZG= \Dom(\Ac)+\ran(R(\gl_0,\Ac)\Bc)$ for some $\gl_0\in\rho(\Ac)$ and denote the $\Lambda$-extension of $\Cc$ by $\CcL$.
The passivity implies that
$\re \iprod{\Ac z+\Bc y}{z}\leq \re \iprod{\Cc z+\Dc y}{y}$ for all $z\in Z$ and $y\in Y$ satisfying $\Ac z+\Bc y\in Z$, and we have $\Dc\geq 0$.
We denote the transfer function of $(\Ac ,\Bc ,\Cc ,\Dc)$ with 
\eq{
  \G(\gl) = \CcL R(\gl,\Ac)\Bc + \Dc, \qquad \gl\in\rho(\Ac).
}
Our assumption $\Dc\geq 0$ simplies the analysis of the admissibility of output feedbacks of the two passive systems~\eqref{eq:plantintro} and~\eqref{eq:sys2intro}. 
  However, many of the results also hold in the situation where $\re \Dc\geq 0$ as long as the appropriate feedback operators remain admissible, which is the case, e.g., 
  if $\norm{\Dc-\Dc^\ast}$  
  is sufficently small.

\section{Stability of Coupled Passive Systems}
\label{sec:CLstab}

In this section we present our main results on the stability of the closed-loop system associated to the power-preserving interconnection of~\eqref{eq:plantintro} and~\eqref{eq:sys2intro}.  
Lemma~\ref{lem:CLreg} in Section~\ref{sec:RORP} shows that the system operator $A_e$ of the closed-loop system 
\eq{
  \dot{x}_e (t) &= A_ex_e(t), \qquad x_e(0)=x_{e0}=(x_0,z_0)^T\in X_e
}
is given by 
\begin{subequations}
  \label{eq:CLsysop}
  \eqn{
    &\hspace{.5cm} A_e = \pmat{A-B\Dc Q_1\CL&BQ_2\CcL\\-\Bc Q_1\CL & \Ac -\Bc Q_1D\CcL},\\
    \Dom(A_e) &= \biggl\{\pmat{x\\z}\in 
      \XB\times \ZG
      \biggm|
      \begin{array}{l}
	(A-B\Dc Q_1\CL)x+BQ_2\CcL z\in X\\
	-\Bc Q_1\CL x+ (\Ac -\Bc Q_1D\CcL)z\in Z
      \end{array}
      \biggr\},
  }
\end{subequations}
where $Q_1 = (I+D\Dc )\inv$ and $Q_2=(I+\Dc D)\inv$, 
and that
$A_e$ generates a strongly continuous contraction semigroup $T_e(t)$ on $X_e$. 

\begin{remark}
  \label{rem:Dcrole}
  Our results assume that~\eqref{eq:plantintro} is stable and its transfer function $P(\gl)$ satisfies certain additional conditions. However, the results are also immediately applicable when~\eqref{eq:plantintro} is unstable but can be stabilized with a suitable output feedback. 
  Indeed, if $\Dc>0$, we can write $\Dc = \Dcone+\Dctwo$ with $\Dcone \geq 0$ and $\Dctwo>0$.
  Lemma~\ref{lem:Repostoinv}(d) implies that $u(t)=-\Dctwo y(t)$ with $\Dctwo>0$ is an admissible feedback for $\PARsys$
and the resulting system $\PARsysS=(A-B\Dctwo Q_1^S \CL,BQ_2^S,Q_1^S\CL,Q_1^SD)$ with
$Q_1^S=(I+D\Dctwo)\inv$ and $Q_2^S=(I+\Dctwo D)\inv$ is regular~\cite{Wei94}.
A direct computation shows that
\eq{
  A_e
  &=\pmat{A^S-B^S\Dcone Q_3\CL^S&B^SQ_4\CcL\\-\Bc Q_3\CL^S & \Ac -\Bc Q_3D^S\CcL}.
}
Since this operator has exactly the same form as the original $A_e$, in each of our results it is possible to replace  $\PARsys$ with the stabilized system $\PARsysS$, the transfer function $P(\gl)$ with $P_S(\gl)=\CL^S R(\gl,A^S)B^S+D^S$, and the feedthrough operator $\Dc\geq 0$ with $\Dcone\geq 0$.
It is important to note that if $P(\gl)$ is invertible and $\re P(\gl)\geq 0$ for some $\gl\in \rho(A)$, then for any $\Dctwo>0$ we have $\re P_S(\gl)>0$.
\end{remark}

\subsection{Strong Stability}

The following theorem presents sufficient conditions for the strong stability of the closed-loop system.

\begin{theorem}
  \label{thm:CLstabstr}
  Assume $\PARsys$ is
  passive and strongly stable in such a way that $i\R\subset \rho(A)$.
Moreover, assume $\PARcontr$ is passive,
$\Dc\geq 0$,
and the following hold for some $\SIndset\subset \Z$.
  \begin{itemize}
    \setlength{\itemsep}{.5ex}
    \item[\textup{(1)}] 
      $\gs(\Ac )\cap i\R=\set{i\gw_k}_{k\in \SIndset}$ and $\re P(i\gw_k)>0$ for all $k\in\SIndset$.
    \item[\textup{(2)}]
  $I+P(i\gw)\G(i\gw)$ has a bounded inverse for every $\gw\in \R\setminus \set{\gw_k}_{k\in\SIndset}$ for which $\re \G(i\gw)$ is not boundedly invertible.
    \item[\textup{(3)}]
       $\set{i\gw_k}_{k\in \SIndset}\subset \rho(\Ac  -\Bc D_0(I+\Dc D_0)\inv \CcL)$ whenever $\re D_0>0$.
\end{itemize}
Then $i\R\subset \rho(A_e)$ and the closed-loop system is strongly stable.

  Assume in addition that $\SIndset\subset \Z$ is finite, $\PARsys$ is exponentially stable,
  and $\sup_{\abs{\gw}\geq \gwlim}\norm{R(i\gw,\Ac )}<\infty$ for some $\gwlim>0$.
  If we either have  
  $\limsup_{\abs{\gw}\to \infty}\norm{\G(i\gw)P(i\gw)}<1$,
  or if
$\re P(i\gw)\geq \eta(\gw)\geq 0$ and $\re \G(i\gw)\geq \dc(\gw)\geq 0$ so that 
  $\eta(\gw)+\dc(\gw)\geq \eta_0>0$ for some constant $\eta_0>0$ and for all sufficiently large $\abs{\gw}$,
  then the closed-loop system is exponentially stable.
\end{theorem}

\begin{proof}
  We begin by showing that $i\R\subset \rho(A_e)$. 
 Since the semigroup generated by $A_e$ is uniformly bounded by Lemma~\ref{lem:CLreg}, the strong stability of $T_e(t)$ then follows from the Arendt--Batty--Lyubich--V\~{u} Theorem~\cite{AreBat88,LyuVu88}.

Lemma~\ref{lem:Repostoinv}(d) implies that $u(t)=-\Dc y(t)$ is an admissible output feedback for $\PARsys$, and by~\cite{Wei94} the resulting system  $(A^{cl},B^{cl},\CL^{cl},D^{cl})=(A-B\Dc Q_1\CL,$ $BQ_2,Q_1\CL,Q_1D)$ is regular. The assumption $i\R\subset \rho(A)$ and Lemma~\ref{lem:Afbreg} imply $i\R\subset \rho(A^{cl})$,
and by Lemma~\ref{lem:Repostoinv}(d) the transfer function $P_{cl}(\gl)$ is given by $P_{cl}(i\gw)=P(i\gw)(I+\Dc P(i\gw))\inv$ for all $\gw\in\R$. 
If $\gw\in\R$ and if we denote $R_{i\gw} = R(i\gw,A^{cl})$, then $i\gw-A_e$ has a bounded inverse given by
\eq{
  R(i\gw,A_e) 
  = \pmat{R_{i\gw}-R_{i\gw} B^{cl} \CcL S_A(i\gw)\inv \Bc  \CL^{cl} R_{i\gw}&R_{i\gw} B^{cl} \CcL S_A(i\gw)\inv\\-S_A(i\gw)\inv \Bc  \CL^{cl} R_{i\gw} &S_A(i\gw)\inv}
}
provided that the Schur complement 
  \eq{
    S_A(i\gw)
    &= i\gw-\Ac  + \Bc  D^{cl} \CcL + \Bc  \CL^{cl} R(i\gw,A^{cl})B^{cl} \CcL \\
    &= i\gw-\Ac  + \Bc  P(i\gw)(I+\Dc P(i\gw))\inv \CcL
  }
with domain $\Dom(S_A(i\gw)) = \setm{z\in \Dom(\CcL)}{S_A(i\gw)z\in Z}$ 
has a bounded inverse.
If $\gw=\gw_n$ for some $n\in\SIndset$, then  $\re P(i\gw_n)>0$ and assumption~(3)
imply that $S_A(i\gw_n)$ is boundedly invertible. Thus $\set{i\gw_k}_{k\in\SIndset}\subset \rho(A_e)$.

Now let $\gw\in \R\setminus \set{\gw_k}_{k\in\SIndset}$. If $\re G(i\gw)\not>0$, then $I+\G(i\gw)P(i\gw)$ is invertible by condition~(2) of the theorem. By Lemma~\ref{lem:Repostoinv}(a) the same is also true if $\re G(i\gw)>0$, since $I+\G(i\gw)P(i\gw)=\G(i\gw)(\G(i\gw)\inv + P(i\gw))$.
Because
\eq{
I+\Dc P(i\gw)+  \CcL R(i\gw,\Ac ) \Bc P(i\gw)
=  I+ \G(i\gw)P(i\gw),
}
Lemma~\ref{lem:Woodbury} implies that
$S_A(i\gw)$ has a bounded inverse 
  \eqn{
  \label{eq:SAinvWoodbury}
  S_A(i\gw)\inv 
  &= R(i\gw,\Ac ) \bigl[I \hspace{-.2ex} - \hspace{-.2ex}  \Bc P(i\gw)(I \hspace{-.2ex} + \hspace{-.2ex} \G(i\gw) P(i\gw) )\inv
  \CcL  R(i\gw,\Ac )\bigl]. \hspace{-1.6ex}
  }
  Thus $i\gw\in \rho(A_e)$ also for all $\gw\in\R \setminus\set{\gw_k}_{k\in\SIndset}$. 
Since the semigroup $T_e(t)$ is contractive, the closed-loop system is strongly stable.

  Finally, assume that $\SIndset\subset \Z$ is finite, $\PARsys$ is exponentially stable, and $\sup_{\abs{\gw}\geq \gwlim}\norm{R(i\gw,\Ac )}<\infty$ for some $\gwlim>0$.
   The stability and regularity of $\PARsys$ imply that the norms $\norm{R(\cdot,A)}$, $\norm{R(\cdot,A)B}$, $\norm{\CL R(\cdot,A)}$, and $\norm{P(\cdot)}$ are uniformly bounded on $i\R$.
  Similarly the regularity of the controller implies that 
  $\norm{R(i\gw,\Ac )}$, $\norm{R(i\gw,\Ac )\Bc}$, $\norm{\CcL R(i\gw,\Ac )}$, and $\norm{\CcL R(i\gw,\Ac ) \Bc}$ are uniformly bounded with respect to $\gw\in\R$ with $\abs{\gw}\geq \gwlim$.
  If
  $\limsup_{\abs{\gw}\to \infty}\norm{\G(i\gw)P(i\gw)}<1$
  the norms $\norm{P(i\gw)(I+\G(i\gw)P(i\gw))\inv}$ are uniformly bounded for large $\abs{\gw}$.
  On the  other hand, if
  $\eta(\gw)+\dc(\gw)\geq \eta_0>0$, then Lemma~\ref{lem:Repostoinv}(b) implies $\norm{P(i\gw)(I+\G(i\gw)P(i\gw))\inv}\lesssim \eta_0\inv $.
   Thus~\eqref{eq:SAinvWoodbury} implies that
   $\norm{R(i\gw,A_e)}$ is uniformly bounded for large $\abs{\gw}$. 
  Since $i\R\subset \rho(A_e)$ and $T_e(t)$ is contractive, the 
  closed-loop system is exponentially stable.
\end{proof}

\begin{remark}
  \label{rem:Dc0}
  Condition~(2) is in particular satisfied if
$\re \G(i\gw)>0$ for all 
 $\gw\in \R\setminus \set{\gw_k}_{k\in\SIndset}$. Moreover, if  
 $\re \G(i\gw)\geq \dc>0$ for some constant $\dc>0$ and for all $\gw\in \R\setminus \set{\gw_k}_{k\in\SIndset}$, then 
$\norm{P(i\gw)(I+\G(i\gw) P(i\gw) )\inv}\leq \dc\inv$
 for all 
 $\gw\in \R\setminus \set{\gw_k}_{k\in\SIndset}$ by Lemma~\ref{lem:Repostoinv}(b).

  The proof of Theorem~\ref{thm:CLstabstr} can also be adapted to show that if $\re P(i\gw)>0$ for all $\gw\in\R$, then
   $T_e(t)$ is strongly stable and $i\R\subset \rho(A_e)$ even without assumption~(2).
  Indeed, if $\gw\in \R\setminus \set{\gw_k}_{k\in\SIndset}$ 
  and $\re P(i\gw)>0$, 
  then Lemma~\ref{lem:Repostoinv}(a) implies that both
  $P(i\gw)$ and $I+\G(i\gw) P(i\gw) = (P(i\gw)\inv + \G(i\gw))P(i\gw)$ are boundedly invertible, and $S_A(i\gw)$ has the bounded inverse given by the formula~\eqref{eq:SAinvWoodbury}. 
  Thus we again have
   $i\gw\in\rho(A_e)$.
   Lemma~\ref{lem:Repostoinv}(b) also shows that if $\eta(\gw)>0$  is such that $\re P(i\gw)\geq \eta(\gw)>0$, then 
  $\norm{P(i\gw)(I+\G(i\gw) P(i\gw))\inv} \leq \eta(\gw)\inv \norm{P(i\gw)}^2$.
\end{remark}

  The following lemma provides a sufficient condition for the assumption~(3) in Theorem~\ref{thm:CLstabstr} for isolated spectral points under a suitable observability property.
\begin{lemma}
  \label{lem:CLstabstrongsuff}
Assume $\PARcontr$ is passive with
$\Dc\geq 0$.
Assume further that $i\gw_k\in \gs(A_c)$ is an isolated spectral point 
and $A_c$ has a spectral decomposition $A_c=A_c^0+A_c^c$ according to $Z=\ker(i\gw_k-A_c)\oplus \ker(i\gw_k-A_c)^\perp$ so that $i\gw_k\in\rho(A_c^c)$, 
and there exists $\gg>0$ such that $\norm{\CcL z}\geq \gg \norm{z}$ for all $z\in \ker(i\gw_k-A_c)$. 
Then $i\gw_k\in \rho(A_c-B_cD_0(I+D_cD_0)\inv \CcL)$ for any $D_0\in \Lin(U)$ with $\re D_0>0$.
\end{lemma}

\begin{proof}
  Let $D_0\in \Lin(U)$ be such that $\re D_0\geq d_0>0$ and denote $D_1=D_0(I+\Dc D_0)\inv$. 
  Due to the passivity of $\PARcontr$ and~\cite[Cor. 4.3.2]{AreBat11book} we have 
   $i\gw_k\in\gs(\Ac - \Bc D_1\CcL)$ provided that
   $\norm{(i\gw_k-\Ac + \Bc D_1\CcL)z}\geq c \norm{z}$ for some constant $c>0$ and for all  
    $z\in \Dom(A_c-B_cD_1\CcL)\subset \ZG$.
    Let $z\in \Dom(A_c-B_cD_1\CcL)$ and denote $y=(i\gw_k-\Ac + \Bc D_1\CcL)z$.
The passivity of $(\Ac ,\Bc,\CcL ,\Dc)$
  implies 
  \eq{
    \re \iprod{y}{z}
    &=-\re \iprod{\Ac z+ \Bc (-D_1\CcL z)}{z}
    \geq \re \iprod{\CcL z- \Dc D_1\CcL z}{D_1\CcL z}\\
    &= \re \iprod{(I+\Dc D_0)\inv \CcL z}{D_0(I+\Dc D_0)\inv \CcL z}\\
   & \geq   d_0\norm{I+\Dc D_0}^{-2} \norm{\CcL z}^2 .
  }
  Thus $\norm{\CcL z}^2\lesssim \norm{z}\norm{y}$.
  Write $z=z^k+z^c$ according to the decomposition $Z=\ker(i\gw_k-A_c)\oplus \ker(i\gw_k-A_c)^\perp$.
  If we apply $R_1=R(i\gw_k+1,A_c)$ to both sides of  $y=(i\gw_k-A_c+B_cD_1\CcL)z$ and use 
 $R_1z^k\in \ker(i\gw_k-A_c)$ we obtain
 \eqn{
   \label{eq:CLstabstrongsuff}
  (i\gw_k-A_c^c)R_1 z^c = R_1 y-R_1B_cD_1\CcL z .
} 
Since
$R_1B_c \in \Lin(U,Z)$ and 
$i\gw_k-A_c^c$ is boundedly invertible by assumption, 
we have $\norm{R_1z^c}^2\lesssim \norm{(i\gw_k-A_c^c)R_1z^c}^2\lesssim\norm{y}^2 + \norm{\CcL z}^2 \lesssim \norm{y}^2 + \norm{z} \norm{y}$.
Moreover, 
$(i\gw_k-A_c)R_1z^c=z^c-R_1z^c$ and  
$\norm{z^c}\leq \norm{z}$ together with~\eqref{eq:CLstabstrongsuff} further imply
 \eq{
   \norm{z^c}^2
   &= \norm{R_1 z^c+ R_1y-R_1B_cD_1\CcL z}^2\\
   &\lesssim \norm{R_1z^c}^2 + \norm{y}^2 + \norm{\CcL z}^2
   \lesssim  \norm{y}^2 + \norm{z} \norm{y}\\
   \norm{\CcL z^c}^2
   &= \norm{\CcL R_1 (z^c+y)-\CcL R_1B_cD_1\CcL z}^2 \\
   &\lesssim \norm{z^c}^2 + \norm{y}^2 + \norm{\CcL z}^2
   \lesssim  \norm{y}^2 + \norm{z} \norm{y}.
 }
 Finally, since $\norm{z^k}^2\leq \gg^{-2}\norm{\CcL z^k}^2\lesssim \gg^{-2}(\norm{\CcL z}^2+\norm{\CcL z^c}^2)\lesssim \norm{y}^2 + \norm{z}\norm{y}$, we have
 $  \norm{z}^2 = \norm{z^k}^2 + \norm{z^c}^2\lesssim \norm{y}^2 + \norm{z}\norm{y}$, and thus also $\norm{z}\lesssim \norm{y}$.
\end{proof}

\subsection{Exponential Stability}

The following theorem presents sufficient conditions for exponential stability of the closed-loop system.
The transfer function
$P(i\gw)$ is allowed to be non-invertible for some values $\gw\in\R$ (i.e., the system $\PARsys$ may have ``transmission zeros'' on $i\R$), but  such points must be uniformly disjoint from the spectrum of $\Ac$. 
It should be noted that the result also remains valid if the conditions are satisfied for $\Omega=\R$.
Condition~(2) is in particular satisfied if $\re G(i\gw)\geq \dc>0$ for some constant $\dc>0$ and for all $\gw\in \R\setminus \Omega$.
Here \keyterm{exponential stabilizability} and \keyterm{exponential detectability} of a regular linear system are defined as in~\citel{Reb93}{Def.~1.4--1.5} and~\citel{WeiCur97}{Sec.~III}.

\begin{theorem}
  \label{thm:CLstabexp}
    Assume $\PARsys$ is passive and exponentially stable,  $\re D>0$,
    and 
    there exist $\Omega\subset \R$ and $\eta_0>0$ such that $\re P(i\gw)\geq \eta_0>0$ for all $\gw \in \Omega $.
Moreover, assume $\PARcontr$ is passive,
$\Dc\geq 0$,
and the following hold.
  \begin{itemize}
    \setlength{\itemsep}{.5ex}
    \item[\textup{(1)}]
      $\gs(\Ac)\cap i\R\subset i\Omega$
      and  $\sup_{\gw\in\R\setminus\Omega}\norm{R(i\gw,\Ac)}<\infty$.
    \item[\textup{(2)}] 
	Let $\eta(\cdot),d_c(\cdot):\R\setminus \Omega\to [0,1]$ be such that
	$\re P(i\gw)\geq \eta(\gw)\geq 0$ and $\re \G(i\gw)\geq \dc(\gw)\geq 0$ for all $\gw\in \R\setminus \Omega$.
	Assume there exist $0<\gd<1$ and $\eta_1>0$ such that
	for each $\gw\in \R\setminus \Omega$ either $\norm{\G(i\gw)P(i\gw)}\leq \gd <1$ or
	$\eta(\gw)+\dc(\gw)\geq \eta_1>0$.
\item[\textup{(3)}]  
  The system $(\Ac,\Bc,\CcL,\Dc)$ is exponentially stabilizable and detectable.
  \end{itemize}
  Then the closed-loop system is exponentially stable.
\end{theorem}

\begin{proof}
  Our aim is to show $i\R\subset \rho(A_e)$ and $\sup_{\gw\in\R} \norm{R(i\gw,A_e)}<\infty$.
First let $\gw\in \R \setminus \Omega$.
The proof of Theorem~\ref{thm:CLstabstr} shows that  $S_A(i\gw)$ has an inverse
  \eq{
    S_A(i\gw)\inv 
    &= R(i\gw,\Ac) \bigl[I -  \Bc  P(i\gw)(I+\G(i\gw) P(i\gw) )\inv
    \CcL  R(i\gw,\Ac)\bigl].
  }
  If $\norm{\G(i\gw)P(i\gw)}\leq \gd<1$, then $\norm{P(i\gw)(I+\G(i\gw) P(i\gw) )\inv}\leq \norm{P(i\gw)}/(1-\gd)$, and if $\eta(\gw)+\dc(\gw)\geq \eta_1>0$, 
  Lemma~\ref{lem:Repostoinv}(b) implies $\norm{P(i\gw)(I+\G(i\gw)P(i\gw))\inv}\leq \eta_1\inv \max \set{1,\norm{P(i\gw)}} $.
  Assumption (1) and the admissiblity of $\Bc $ and $\Cc$ imply $i\R\setminus i\Omega\subset \rho(A_e)$ and $\sup_{\gw\in\R\setminus \Omega}\norm{R(i\gw,A_e)}<\infty$.

It remains to consider $\gw\in\Omega$.
We decompose  $D$ into two parts $D=\mu D + \nu D$ with $\mu\in (0,1)$ and $\nu = 1-\mu$ in such a way that the first part stabilizes $\PARcontr$ exponentially and the second part can be used to show closed-loop stability.
Indeed, for any $\mu\in(0,1)$ the transfer function of the system $(\Ac^\mu,\Bc^\mu,\CcL^\mu,\Dc^\mu)$ obtained from $(\Ac,\Bc ,\CcL,\Dc)$ with the admissible output feedback $\uc(t)=-\mu D \yc(t)$ is given by $G(\gl)(I+\mu D G(\gl))\inv$. Since $\re D>0$, this transfer function is uniformly bounded on $\C_+$ by Lemma~\ref{lem:Repostoinv}(b), and since 
  $(\Ac^\mu,\Bc^\mu,\CcL^\mu,\Dc^\mu)$ is exponentially stabilizable and detectable due to assumption (3), the semigroup generated by $\Ac^\mu$ is exponentially stable~\citel{Reb93}{Cor. 1.8}. 

For all sufficiently small $\mu \in (0,1)$
the transfer function
$P_\nu(\gl)$ of $(A,B,\CL,$ $\nu D)$ satisfies $\re P_\nu(i\gw)\geq  \tilde{\eta}_0>0$ for some constant $\tilde{\eta}_0>0$ and for all $\gw\in\Omega$. 
Since $\Dc^\mu = \Dc (I+\mu D\Dc)\inv$, Lemmas~\ref{lem:Repostoinv} and~\ref{lem:IDPinvert} imply that we can choose $\mu\in (0,1)$
so that 
$I+\nu D \Dc^\mu  $
and $I+P_\nu(i\gw)\Dc^\mu  $ for all $\gw\in\Omega$ are invertible, and $\sup_{\gw\in\Omega} \norm{(I+P_\nu(i\gw)\Dc^\mu)\inv}<\infty$.
Thus $u(t)=-\Dc^\mu y(t)$ is an admissible output feedback for $(A,B,\CL,\nu D)$.
Denoting the resulting regular linear system with $(A^\mu,B^\mu,\CL^\mu,D^\mu) = (A-B\Dc^\mu Q_5^\mu \CL,BQ_6^\mu,Q_5^\mu \CL,\nu Q_5^\mu D)$ where
$Q_5^\mu = (I+\nu D\Dc^\mu)\inv$ and $Q_6^\mu = (I+\nu \Dc^\mu D)\inv$, we can write
\eq{
  A_e 
  & \hspace{-.3ex}=\hspace{-.3ex} \pmat{A-B\Dc^\mu Q_5^\mu \CL&BQ_6^\mu \CcL^\mu\\-\Bc^\mu Q_5^\mu \CL & \Ac^\mu-\nu\Bc^\mu Q_5^\mu D \CcL^\mu}
  \hspace{-.3ex}=\hspace{-.3ex} \pmat{A^\mu&B^\mu \CcL^\mu\\-\Bc^\mu \CL^\mu & \Ac^\mu-\Bc^\mu D^\mu \CcL^\mu}\hspace{-.3ex}.
}
Similarly as in Lemma~\ref{lem:Afbreg} we can show that $\sup_{\gw\in\Omega}\norm{R(i\gw,A^\mu)}<\infty$
and the transfer function
of $(A^\mu,B^\mu,\CL^\mu,D^\mu)$ satisfies $P_\mu(i\gw) = P_\nu(i\gw)(I+\Dc^\mu P_\nu(i\gw))\inv$ for all $\gw\in \Omega$. 
The transfer function of $(\Ac^\mu,\Bc^\mu,\CcL^\mu,\Dc^\mu)$ is denoted by $\Gmu(\gl)$.

Let $\gw\in \Omega$.
If we denote $R_{i\gw}^\mu = R(i\gw,A^\mu)$, then $i\gw-A_e$ has a bounded inverse 
\eq{
  R(i\gw,A_e) 
  = \pmat{R_{i\gw}^\mu-R_{i\gw}^\mu B^\mu \CcL^\mu S_A^\mu(i\gw)\inv \Bc^\mu \CL^\mu R_{i\gw}^\mu&R_{i\gw}^\mu B^\mu \CcL^\mu S_A^\mu(i\gw)\inv\\-S_A^\mu(i\gw)\inv \Bc^\mu \CL^\mu R_{i\gw}^\mu &S_A^\mu(i\gw)\inv}
}
provided that the Schur complement 
\eq{
  S_A^\mu(i\gw)
  &= i\gw-\Ac^\mu + \Bc^\mu D^\mu \CcL^\mu + \Bc^\mu \CL^\mu R(i\gw,A^\mu)B^\mu \CcL^\mu \\
  &= i\gw-\Ac^\mu + \Bc^\mu P_\nu(i\gw)(I+\Dc^\mu P_\nu(i\gw))\inv \CcL^\mu
}
has a bounded inverse.
If $S_A^\mu(i\gw)$ is boundedly invertible for all $\gw\in\Omega$, then
the regularity of $(A^\mu,B^\mu,\CL^\mu,D^\mu)$
and  $\sup_{\gw\in\Omega}\norm{R(i\gw,A^\mu)}<\infty$ 
imply $\sup_{\gw\in\Omega}\norm{R(i\gw,A_e)}<\infty$ provided that  
 $\norm{S_A^\mu(i\gw)\inv}$, $\norm{S_A^\mu(i\gw)\inv\Bc^\mu}$, $\norm{\CcL^\mu S_A^\mu(i\gw)\inv}$, and $\norm{\CcL^\mu S_A^\mu(i\gw)\inv \Bc^\mu}$ are uniformly bounded with respect to $\gw\in\Omega$.

Let $\gw\in\Omega$ be arbitrary.
Since $\re P_\nu(i\gw)\geq \tilde{\eta}_0>0$ and $\re \Gmu(i\gw)\geq 0$,
Lemma~\ref{lem:Repostoinv} implies that $P_\nu(i\gw)$ and $I+\Gmu(i\gw)P_\nu(i\gw) = (P_\nu(i\gw)\inv + \Gmu(i\gw)) P_\nu(i\gw)$ are boundedly invertible.
Therefore the same is true for
\eq{
I+\Dc^\mu P_\nu(i\gw)+  \CcL^\mu R(i\gw,\Ac^\mu) \Bc^\mu P_\nu(i\gw)
= I+ \Gmu(i\gw)P_\nu(i\gw).
}
Lemma~\ref{lem:Woodbury} implies that
$S_A^\mu(i\gw)$ has a bounded inverse 
  \eq{
  S_A^\mu(i\gw)\inv 
  = R(i\gw,\Ac^\mu) \bigl[I -  \Bc^\mu P_\nu(i\gw)(I+\Gmu(i\gw) P_\nu(i\gw) )\inv 
   \CcL^\mu  R(i\gw,\Ac^\mu)\bigl],
  }
where $\norm{P_\nu(i\gw)(I+\Gmu(i\gw) P_\nu(i\gw) )\inv}\leq \norm{P_\nu(i\gw)}^2/\tilde{\eta}_0$.
  Thus $i\gw\in \rho(A_e)$. 
  Since $\sup_{\gw\in\R}\norm{P_\nu(i\gw)}<\infty$ and $(\Ac^\mu,\Bc^\mu,\CcL^\mu,\Dc^\mu)$ is regular and exponentially stable, 
the norms $\norm{S_A^\mu(i\gw)\inv}$, $\norm{S_A^\mu(i\gw)\inv\Bc^\mu}$, $\norm{\CcL^\mu S_A^\mu(i\gw)\inv}$, and $\norm{\CcL^\mu S_A^\mu(i\gw)\inv \Bc^\mu}$ are uniformly bounded with respect to $\gw\in\Omega$.
This further implies that $\sup_{\gw\in\Omega}\norm{R(i\gw,A_e)}<\infty$, and the closed-loop system is exponentially stable.
\end{proof}

Since both $\PARsys$ and $\PARcontr$ are exponentially stabilizable in Theorem~\ref{thm:CLstabexp}, the exponential closed-loop stability could alternatively be studied using~\citel{WeiCur97}{Prop. 4.6}.

\subsection{Non-Uniform Closed-Loop Stability}

In this section we introduce conditions for polynomial and non-uniform stability of the closed-loop system
in the case where $\Ac$ is diagonal.
  In addition, our main result can be used as an alternative to Theorem~\ref{thm:CLstabexp} in showing exponential closed-loop stability.
The closed-loop system is said to be non-uniformly stable when $T_e(t)$ is uniformly bounded and $i\R\subset \rho(A_e)$ but the norms $\norm{R(i\gw,A_e)}$ are not bounded with respect to $\gw\in\R$. 
If $M_R(\cdot)$ is a continuous non-decreasing function such that
$\norm{R(i\gw,A_e)}\leq M_R(\abs{\gw})$, then
there exist $M_e,c,t_0>0$ such that
\eqn{
\label{eq:RORPnonuniformstate}
\norm{T_e(t)x_{e0}}\leq \frac{M_e}{M_T(t)} 
\norm{A_e x_{e0}}
\qquad \forall x_{e0}\in \Dom(A_e) , ~ t\geq t_0,
}
where the continuous non-decreasing function $M_T(\cdot):[0,\infty)\to (0,\infty)$ is determined by the results in~\cite{BatDuy08,BorTom10,RozSei19}. In particular, if $M_R(\gw)\lesssim 1+\gw^\ga$ for some $\ga>0$, we can choose $M_T(t)=t^{1/\ga}$~\cite{BorTom10}, and if $M_R(\gw)\lesssim 1+e^{\ga\gw }$
for some $\ga>0$, then
we can choose $M_T(t)=\log(t)/\ga$~\cite[Ex. 1.6]{BatDuy08}.

In this section we assume 
$(A_c,B_c,\CcL,D_c)$ is regular and passive with $\Dc\geq 0$ on a Hilbert space
$Z=\bigotimes_{k\in\SIndset}Z_k$ with norm $\norm{(z_k)_k}_{Z}^2 = \sum_{k\in\SIndset} \norm{z_k}_{Z_k}^2$ where
 $Z_k$ are Hilbert and  $\SIndset\subset \Z$ is infinite. 
We assume $A_c$ has the structure
  \eqn{
  \label{eq:Acdiag}
    \Ac  &= 
    \diag(i\gw_k I_{Z_k})_{k\in \SIndset}, ~\;
    \Dom(\Ac)= \Setm[big]{(z_k)_k\in Z}{\sum_{k\in\SIndset}\abs{\gw_k}^2\norm{z_k}_{Z_k}^2<\infty},
    \hspace{-1.7ex}
  }
  where $\gw_k\neq \gw_l$ for $k\neq l$ and $\set{\gw_k}_k$ has no finite accumulation points.
  Since $A_c$ is skew-adjoint, the operators $B_c\in \Lin(Y,Z_{-1})$ and $\Cc\in \Lin(Z_1,Y)$ are formally adjoint, i.e., $\iprod{\Bc u}{z}_{-1,1}=\iprod{u}{\Cc z}$ for all $z\in \Dom(A_c)$ and $u\in Y$, and thus 
\eq{
  \Bc u= (B_{ck} u)_{k\in \SIndset}, \qquad  \mbox{and} \qquad
  \Cc z=\sum_{k\in \SIndset} B_{ck}^\ast z_k, \quad z=(z_k)_{k\in\SIndset}\in \Dom(A_c)
}
for some $B_{ck} \in \Lin(Y,Z_k)$. 
Our main result uses \keyterm{wavepackets} of $A_c$~\citel{TucWei09book}{Sec. 6.9}.

\begin{definition}
  \label{def:wavepacket}
  Let $\gw\in \R$ and $\gd>0$.
  An element $z=(z_k)_{k\in\SIndset}\in Z$ is a \emph{$(\gw,\gd)$-wavepacket of $A_c$} if 
   $z_k=0$ for those $k\in\SIndset$ for which $\abs{\gw-\gw_k}\geq \gd$.
\end{definition}

The following theorem is the main result of this section.
  The role of $\Omega_\eps\subset \R$ is to show that 
  only the behaviour of $\re P(i\gw)$ near $\gs(\Ac)=\set{i\gw_k}_{k\in\SIndset}$ affects the asymptotic growth of $\norm{R(i\gw,A_e)}$.
By~\citel{Mil12}{Cor. 2.17} $\gd(\cdot)$ and $\gg(\cdot)$ can be chosen as constant functions if and only if $(A_c,B_c)$ is exactly controllable.
 The assumption that $M_R(\cdot):[0,\infty)\to(0,\infty)$ has ``positive increase'' means that there exists $\ga,c,\gw_0>0$ such that $M_R(\gl\gw)\geq c\gl^\ga M_R(\gw)$ for all $\gl>0$ and $\gw\geq \gw_0$~\citel{RozSei19}{Sec. 2}, and this  condition is in particular satisfied if $M_R(\cdot)$ grows polynomially or exponentially.
  The estimation of $\norm{S_A(i\gw)\inv}$ in the proof
  extends techniques developed in~\cite{ChiPau19arxiv}.

  \begin{theorem}
    \label{thm:CLstabnonuniform}
    Assume $\PARsys$ 
    is passive and exponentially stable and the system $(\Ac,\Bc,\CcL,\Dc)$ is 
    passive with $\Ac$ of form~\eqref{eq:Acdiag} and $\Dc\geq 0$.
    Assume further that  
    condition~\textup{(2)} of Theorem~\textup{\ref{thm:CLstabexp}} is satisfied for $\Omega=\Omega_\eps:= \setm{\gw\in \R}{\exists k\in\SIndset:\abs{\gw-\gw_k}<\eps}$ with some $\eps>0$, and that there exist continuous non-increasing functions $\eta(\cdot),\gd(\cdot),\gg(\cdot):\R_+\to(0,1]$ with the following properties.
    \begin{itemize}
	\setlength{\itemsep}{.5ex}
      \item $\re P(i\gw)\geq \eta(\abs{\gw})$ for all $\gw\in \Omega_\eps$.
      \item $\norm{\Cc z}\geq \gg(\abs{\gw}) \norm{z}$ for every $\gw\in \R$ and every  $(\gw,\gd(\abs{\gw}))$-wavepacket $z$ of~$A_c$.
    \end{itemize}
    Then $T_e(t)$ is strongly stable, $i\R\subset \rho(A_e)$, and
    \eq{
      \norm{R(i\gw,A_e)}\leq M_R(\abs{\gw}), \qquad \mbox{where} \quad M_R(\cdot)=M_0 \eta(\cdot)\inv\gg(\cdot)^{-2}\gd(\cdot)^{-2} 
    }
    for some $M_0>0$. Moreover, the following hold.
    \begin{itemize}
      \item[\textup{(a)}] If $\sup_{\gw>0} M_R(\gw)<\infty$, then $T_e(t)$ is exponentially stable. 
      \item[\textup{(b)}] If $M_R(\cdot)$ is strictly increasing and has positive increase,
	then~\eqref{eq:RORPnonuniformstate} holds with $M_T(t)=M_R\inv(ct)$ for some constants $M_e,c,t_0>0$.
      \item[\textup{(c)}] 
	For all other $M_R(\cdot)$,~\eqref{eq:RORPnonuniformstate} holds with $M_T(t)=\Mlog\inv(ct)$ for some $M_e,c,t_0>0$ where
	$\Mlog(\gw) = M_R(\gw)\left( \log(1+M_R(\gw) )+\log(1+\gw)\right)$ for $\gw>0$.
    \end{itemize}
  \end{theorem}

\begin{proof}
    By Theorem~\ref{thm:CLstabstr} and Lemma~\ref{lem:CLstabstrongsuff} the closed-loop system is strongly stable and $i\R\subset \rho(A_e)$.
Once we show $\norm{R(i\gw,A_e)}\leq M_R(\abs{\gw})$ the stability properties of the closed-loop system follow from the characterization of exponential stability (part (a)), from~\citel{RozSei19}{Thm. 1.1} (part (b)), and from~\citel{BatDuy08}{Thm. 1.5} (part (c)).

Since $(A^{cl},B^{cl},\CL^{cl},D^{cl})$ is regular and exponentially stable by Lemma~\ref{lem:Afbreg},
we have from the proof of Theorem~\ref{thm:CLstabstr} that 
for all $\gw\in\R$
\eq{
  \norm{R(i\gw,A_e)}\lesssim \max\bigl\{
    &\norm{S_A(i\gw)\inv},
    \norm{S_A(i\gw)\inv B_c},
    \norm{\CcL S_A(i\gw)\inv},\\
    &\quad\norm{\CcL S_A(i\gw)\inv B_c}\bigr\},
}
where  $S_A(i\gw) = i\gw-\Ac  + \Bc P_{cl}(i\gw) \CcL $ and $P_{cl}(i\gw) = P(i\gw)(I+\Dc P(i\gw))\inv $.
Moreover, \eqref{eq:SAinvWoodbury} and our assumptions 
imply $\sup_{\gw\in\R\setminus \Omega_\eps}\norm{R(i\gw,A_e)}<\infty$ 
similarly as in the proof of Theorem~\ref{thm:CLstabexp}.
Thus it is sufficient to
show that for each
$\gw\in\Omega_\eps$
 the norms
    $\norm{S_A(i\gw)\inv}$,
    $\norm{S_A(i\gw)\inv B_c}$,
    $\norm{\CcL S_A(i\gw)\inv}$,
    $\norm{\CcL S_A(i\gw)\inv B_c}$ are bounded by $M_R(\abs{\gw}) $ for some constant $M_0>0$.

    We begin by showing $\norm{\CcL S_A(i\gw)\inv B_c}\leq M_R(\abs{\gw})$. Formula~\eqref{eq:SAinvWoodbury} implies that for all $\gw\in \Omega_\eps\setminus \set{\gw_k}_k$
    \eq{
      \MoveEqLeft     \CcL  S_A(i\gw)\inv B_c\\
      &=  \CcL R(i\gw,\Ac )B_c \bigl[I  -    (I  +   P(i\gw)\G(i\gw) )\inv P(i\gw)\CcL  R(i\gw,\Ac )B_c\bigl] \\
      &=  (G(i\gw)-\Dc) (I  +   P(i\gw)\G(i\gw) )\inv  (I+P(i\gw)D_c).
    }
    Since $\re P(i\gw)>0$ and $\re G(i\gw)\geq 0$,  $I+P(i\gw)G(i\gw)=P(i\gw)(P(i\gw)\inv +G(i\gw))$ is boundedly invertible by Lemma~\ref{lem:Repostoinv}(a). 
    If we denote $Q(i\gw)=(I+P(i\gw)G(i\gw))\inv $, the above formula and stability of $\PARsys$ implies
    \eq{
      \norm{\CcL S_A(i\gw)\inv B_c}
      &=  \norm{(G(i\gw)-\Dc) Q(i\gw)  (I+P(i\gw)D_c)}\\
      &\lesssim  \norm{G(i\gw)Q(i\gw)}+\norm{ Q(i\gw)}.
    }
    Here $\norm{G(i\gw)Q(i\gw)}\leq \eta(\abs{\gw})\inv$ by Lemma~\ref{lem:Repostoinv}(b). We claim that $\norm{Q(i\gw)}  \lesssim \eta(\abs{\gw})\inv$ for $\gw\in \Omega_\eps\setminus \set{\gw_k}_{k\in\SIndset}$.
    If this is not true, then (considering $Q(i\gw)^\ast$) there exist sequences $(s_n)_n\subset \Omega_\eps\setminus \set{\gw_k}_k$ and $(u_n)_n\subset Y$ with $\norm{u_n}=1$ such that 
 $\eta(\abs{s_n})\inv \norm{(I+G(is_n)^\ast P(is_n)^\ast)u_n} \to 0$ as $n\to\infty$.
Since $\sup_{\gw\in\R} \norm{P(i\gw)}<\infty$, we have that also
  \eq{
    0& \from \frac{1}{\eta(\abs{s_n})}\re \iprod{(I+G(is_n)^\ast P(is_n)^\ast )u_n}{P(is_n)^\ast u_n}
    \geq \frac{\re \iprod{P(is_n)u_n}{u_n}}{\eta(\abs{s_n})} 
  }
as $n\to \infty$, which is impossible since $\eta(\abs{s_n})\inv \re \iprod{P(is_n)u_n}{u_n} \geq 1 $ by assumption. This contradiction shows that the claim holds. 
Thus we have $\norm{\CcL S_A(i\gw)\inv B_c}\lesssim \eta(\abs{\gw})\inv \leq M_R(\abs{\gw})$ for some $M_0>0$ and for all $\gw\in \Omega_\eps\setminus \set{\gw_k}_k$, and by continuity the same estimate holds for every $\gw\in\Omega_\eps$.

    To estimate the norms 
    $\norm{S_A(i\gw)\inv}$,
    $\norm{S_A(i\gw)\inv B_c}$,
    $\norm{\CcL S_A(i\gw)\inv}$,
    let $\gw\in \Omega_\eps$ with $\abs{\gw}\geq 1$ and define $P_{\gw,\gd}=\diag(\gb_k I_{Z_k})_{k\in\SIndset}\in\Lin(Z)$ where $\gb_k=1$ for those $k\in \SIndset$ for which $\abs{\gw-\gw_k}<\gd(\abs{\gw})$ and $\gb_k=0$ otherwise. The operator $P_{\gw,\gd}$ is a spectral projection of $A_c$ associated to the part $\set{i\gw_k}_k\cap (i\gw-i\gd(\abs{\gw}),i\gw+i\gd(\abs{\gw}))$ of its spectrum and $P_{\gw,\gd}z$ is a $(\gw,\gd(\abs{\gw}))$-wavepacket of $A_c$ for every $z\in Z$.
    Let $u\in Y$ and $y\in Z$ be arbitrary and define $z=S_A(i\gw)\inv (B_c u+y)\in \ZG$, i.e., $(i\gw-A_c+B_cP_{cl}(i\gw)\CcL)z=B_c u+y$.

Define $z_0=P_{\gw,\gd}z$, $z_c=z-z_0$, $y_c=P_{\gw,\gd}y$, $y_c=y-y_0$. Similarly decompose $A_c=A_c^0+A_c^c$, $B_c=B_c^0+B_c^c$, and $\CcL=C_c^0 + \CcL^c$ where $A_c^0=A_cP_{\gw,\gd}$, $B_c^0=P_{\gw,\gd}B_c$ and $C_c^0=C_c P_{\gw,\gd}$.
      The diagonal structure of $A_c$ and 
      the decompositions imply
       \eq{
	 &(i\gw-A_c^c)z_c = y_c + B_c^c(u-P_{cl}(i\gw)\CcL z)\\
	 \Rightarrow \quad & 
	 z_c = R(i\gw,A_c^c) y_c + R(i\gw,A_c^c)B_c^c(u-P_{cl}(i\gw)\CcL z)\\
	 \Rightarrow \quad & 
	 \CcL z_c = \CcL R(i\gw,A_c^c) y_c + G_{0c}(i\gw)(u-P_{cl}(i\gw)\CcL z),
       }
       where we have denoted $G_{0c}(i\gw)=\CcL^c R(i\gw,A_c^c)B_c^c$.
       The system $(A_c^c,B_c^c,\CcL^c)$ is regular and due to the diagonal structure of $A_c$ we have
      $\norm{R(i\gw,A_c^c)}\lesssim \gd(\abs{\gw})\inv$. The resolvent identity $R(i\gw,A_c^c)=R(i\gw+1,A_c^c)+R(i\gw,A_c^c)R(i\gw+1,A_c^c)$ and the admissibility of $B_c^c$ and $C_c^c$ further imply
      \eq{
	\norm{R(i\gw,A_c^c) B_c^c}&\lesssim \gd(\abs{\gw})\inv,  \\
	\quad 
	\norm{\CcL R(i\gw,A_c^c)}&\lesssim \gd(\abs{\gw})\inv,  \\
	\quad 
\norm{G_{0c}(i\gw)}
	&\lesssim \gd(\abs{\gw})\inv.
      }
      Since $z_0$ is a $(\gw,\gd(\abs{\gw}))$-wavepacket, we have also 
      \ieq{
	\norm{z_0}\leq \gg(\abs{\gw})\inv \norm{C_c z_0} .
      }
      The above expressions for $z_c$ and $\CcL z_c$ together with $C_c z_0=\CcL z-\CcL z_c$ and  $\sup_{s\in\R}\norm{P_{cl}(is)}<\infty$ (Lemma~\ref{lem:IDPinvert}) therefore imply
      \eq{
	\MoveEqLeft[1.5]\norm{z}^2 
	=  \norm{z_c}^2 + \norm{z_0}^2 
	\leq \norm{z_c}^2+\gg(\abs{\gw})^{-2} \norm{C_c z_0}^2 \\
	&\lesssim \norm{z_c}^2+ \gg(\abs{\gw})^{-2} \norm{\CcL z}^2 + \gg(\abs{\gw})^{-2} \norm{\CcL^c z_c}^2 \\
	& \lesssim \bigl( \norm{R(i\gw,A_c^c)}^2+\gg(\abs{\gw})^{-2} \norm{\CcL R(i\gw,A_c^c)}^2\bigr) \norm{y_c}^2 +\gg(\abs{\gw})^{-2} \norm{\CcL z}^2\\
	&~\; + \bigl(\norm{R(i\gw,A_c^c)B_c^c}^2+\gg(\abs{\gw})^{-2}\norm{G_{0c}(i\gw)}^2\bigr)\bigl(\norm{u}^2 + \norm{P_{cl}(i\gw)}^2 \norm{\CcL z}^2\bigr)\\
	&\lesssim\gg(\abs{\gw})^{-2}\gd(\abs{\gw})^{-2} 
	\bigl( \norm{y}^2 + \norm{u}^2 +  \norm{\CcL z}^2\bigr).
      }

First let $u=0$
to estimate $\norm{S_A(i\gw)\inv}$ and $\norm{\CcL S_A(i\gw)\inv }$. Then $z=S_A(i\gw)y\in \Dom(S_A(i\gw))$.
    The passivity of $(\Ac ,\Bc,\CcL ,\Dc)$
    implies 
    \eq{
      \re \iprod{y}{x} & =-\re \iprod{\Ac z+ \Bc (-P_{cl}(i\gw)\CcL z)}{z}\\
      &\geq \re \iprod{\CcL z- \Dc P_{cl}(i\gw)\CcL z}{P_{cl}(i\gw)\CcL z}\\
      &= \re \iprod{(I+\Dc P(i\gw))\inv \CcL z}{P(i\gw)(I+\Dc P(i\gw))\inv \CcL z}\\
      &\geq \eta(\abs{\gw})  \norm{I+\Dc P(i\gw)}^{-2} \norm{\CcL z}^2 
      \geq \frac{\eta(\abs{\gw})}{M_P^2} \norm{\CcL z}^2,
    }
    where $M_P = 1+\norm{\Dc}\sup_{\gw\in\R}\norm{P(i\gw)}<\infty$, and thus we have $\norm{\CcL z}^2\lesssim \eta(\abs{\gw})\inv \norm{z}\norm{y}$.
      The above estimate for $\norm{z}^2$ (again with $u=0$) together with the scalar inequality $2ab\leq \eps a^2+b^2/\eps$ for $\eps>0$ implies
      \eq{
	\norm{z}^2 
	&\lesssim\gg(\abs{\gw})^{-2}\gd(\abs{\gw})^{-2} \bigl( \norm{y}^2  +  \norm{\CcL z}^2\bigr)\\
	&\lesssim \gg(\abs{\gw})^{-2}\gd(\abs{\gw})^{-2}  \norm{y}^2  +  \eta(\abs{\gw})\inv \gg(\abs{\gw})^{-2}\gd(\abs{\gw})^{-2}\norm{z}\norm{y}\\
	&\leq \gg(\abs{\gw})^{-2}\gd(\abs{\gw})^{-2}  \norm{y}^2  + \frac{\eps}{2}\norm{z}^2 +   \frac{1}{2\eps}\eta(\abs{\gw})^{-2} \gg(\abs{\gw})^{-4}\gd(\abs{\gw})^{-4}\norm{y}^2.
      }
Letting $\eps>0$ be small shows that
$\norm{z}\lesssim \eta(\abs{\gw})\inv\gg(\abs{\gw})^{-2} \gd(\abs{\gw})^{-2} \norm{y}$. 
Since $y\in Z$ was arbitrary,  we have that $\norm{S_A(i\gw)\inv}\leq M_R(\abs{\gw})$ for some $M_0>0$.
Moreover, our earlier estimate $ \norm{\CcL z}^2 \lesssim \eta(\abs{\gw})\inv \norm{z}\norm{y} $ further implies
\eq{
  \norm{\CcL S_A(i\gw)\inv y}^2 
  &= \norm{\CcL z}^2
  \lesssim \eta(\abs{\gw})\inv \norm{z}\norm{y}\\
  &\lesssim \eta(\abs{\gw})^{-2}\gg(\abs{\gw})^{-2} \gd(\abs{\gw})^{-2} \norm{y}^2,
}
and thus $\norm{\CcL S_A(i\gw)\inv}\lesssim \eta(\abs{\gw})\inv \gg(\abs{\gw})\inv\gd(\abs{\gw})\inv
\leq M_R(\abs{\gw}) $ for some $M_0>0$.

Finally, to estimate $\norm{S_A(i\gw)\inv B_c}$, let $y=0$ and let $u\in Y$ be arbitrary.
Now we have $z=S_A(i\gw)\inv B_c u$, and thus $\norm{\CcL z} = \norm{\CcL S_A(i\gw)B_c u}\lesssim \eta(\abs{\gw})\inv \norm{u}$ due to our earlier estimate. Because of this, we also have
\eq{
  \norm{S_A(i\gw)\inv B_c u}^2 
  = \norm{z}^2
  &\lesssim\gg(\abs{\gw})^{-2}\gd(\abs{\gw})^{-2} 
	\bigl(  \norm{u}^2 +  \norm{\CcL z}^2\bigr)\\
	&\lesssim\gg(\abs{\gw})^{-2}\gd(\abs{\gw})^{-2} 
(1+\eta(\abs{\gw})^{-2})  \norm{u}^2 
}
and thus $\norm{ S_A(i\gw)\inv B_c}\lesssim \eta(\abs{\gw})\inv \gg(\abs{\gw})\inv\gd(\abs{\gw})\inv
\leq M_R(\abs{\gw}) $ for some $M_0>0$.
\end{proof}

    In the case where $X=\set{0}$, $A=0\in \Lin(X)$, $B=0\in \Lin(U,X)$, $C=0 \in \Lin(X,U)$, and $D=I\in \Lin(U)$ the operator $S_A(i\gw)$ reduces to $i\gw-\Ac+\Bc (I+\Dc )\inv \CcL$. 
    This way Theorem~\ref{thm:CLstabnonuniform} can also be used to study the non-uniform stability of semigroups generated by operators of the form 
$\Ac-\Bc\Bc^\ast$ 
and $\Ac-\Bc (I+\Dc )\inv \CcL$. This topic is considered in detail in~\cite{ChiPau19arxiv}.

  \begin{remark}
    \label{rem:CLNUunifgap}
    Assume $\set{\gw_k}_{k\in\SIndset}$ has a uniform gap, i.e., $\inf_{k\neq l}\abs{\gw_k-\gw_l}>0$, and 
    $\tilde{\gg}:\R_+\to (0,1]$ is a continuous non-increasing function such that $\inf_{\gw>0}\tilde{\gg}(\gw+\gd_0)/\tilde{\gg}(\gw)>0$ for some  $0<\gd_0<\min\set{1,\frac{1}{2}\inf_{k\neq l}\abs{\gw_k-\gw_l}}$ (so that $\tilde{\gg}(\cdot)$ does not decrease too rapidly). If 
    $\norm{\Bck^\ast z_k}\geq \tilde{\gg}(\abs{\gw_k}) \norm{z_k}$ for all $k\in\SIndset$ and $z_k\in Z_k$, then there exists a constant $0<c\leq 1$ for which 
     the functions $\gg(\cdot)=c \tilde{\gg}(\cdot)$ and $\gd(\cdot)\equiv \gd_0>0$ 
    are such that 
    $\norm{\Cc z}\geq \gg(\abs{\gw}) \norm{z}$ for every $\gw\in \R$ and every $(\gw,\gd(\abs{\gw}))$-wavepacket $z$ of~$A_c$.
  \end{remark}

\section{The Robust Output Regulation Problem}
\label{sec:RORP}

We will now turn our attention to constructing passive controllers of the form \eqref{eq:contrintro} to achieve robust output tracking and disturbance rejection for a passive regular linear system~\eqref{eq:plantfull}. We assume the reference signal $\yref(t)$ and the disturbance signal $\wdist(t)$ are of the form
\eqn{
  \label{eq:yrefwdist}
  \yref(t) = \sum_{k\in\SIndset} \yrefk e^{i\gw_k t},
  \quad \mbox{and} \quad 
 \wdist(t) = \sum_{k\in\SIndset} \wdistk e^{i\gw_k t},
}
with a given set $\set{\gw_k}_{k\in\SIndset}\subset \R$ of distinct frequencies with no finite accumulation points, and $\set{\yrefk}_{k\in\SIndset}\subset Y$ and $\set{\wdistk}_{k\in\SIndset}\subset U_d$.
We use the notation $\distref(t)=(\wdist(t),\yref(t))^T$ and $\distrefk=(\wdistk,\yrefk)^T$.
We consider $\yref(t)$ and $\wdist(t)$ with both finite and infinite number of frequency components, and these two classes of signals are treated separately.
The latter situation is encountered in 
tracking and rejection of nonsmooth periodic signals~\cite{ImmPoh05b}.
If $\SIndset$ is infinite, we assume
  $(\yrefk)_{k\in\SIndset} \in \lp[1](\SIndset;Y)$ and
  $(\wdistk)_{k\in\SIndset} \in \lp[1](\SIndset;U_d)$, which imply that $\yref(t)$ and $\wdist(t)$ are uniformly continuous \keyterm{almost periodic functions}~\citel{AreBat11book}{Def. 4.5.6}.
In the case of real-valued $\yref(t)$ and $\wdist(t)$ we have $\pm\gw_n\in \set{\gw_k}_{k\in\SIndset}$ for all $n\in\SIndset$.

We make the following standing assumption on the system~\eqref{eq:plantfull}.
Here $P_S(\gl)$ is the transfer function of the system $\PARsysS$ obtained from~\eqref{eq:plantfull} with admissible output feedback $u(t)=-\Dctwo y(t)$ with $\Dctwo\geq 0$.
It should be noted that Assumption~\ref{ass:Pkinv} is satisfied for some $\Dctwo\geq 0$ for which $\set{i\gw_k}_k\subset \rho(A^S)$ if and only if it is satisfied for all $\Dctwo\geq 0$ with this property.
In particular, if $i\gw_k\in\rho(A)$ for some $k\in \SIndset$, then $P_S(i\gw_k)$ is invertible if and only if $P(i\gw_k)$ is invertible. 

\begin{assumption}
  \label{ass:Pkinv} There exists $\Dctwo\geq 0$ such that $i\gw_k\in \rho(A^S)$ and $ P_S(i\gw_k)$ is boundedly invertible for all $k\in\SIndset$.
\end{assumption}

We define the \keyterm{regulation error} as $e(t)=\yref(t)-y(t)$. 
Our aim is to choose $\PARcontr$ in such a way that $e(t)$ converges to zero in a suitable sense as $t\to\infty$.
The closed-loop system consisting of~\eqref{eq:plantfull} and the controller~\eqref{eq:contrintro} with state $x_e(t)=(x(t),z(t))^T$ on $X_e=X\times Z$ is of the form
\begin{subequations}
  \label{eq:CLsys}
  \eqn{
    \dot{x}_e(t)&= A_e x_e(t) + B_e \distref(t), \qquad x_e(0)=x_{e0}=(x_0,z_0)^T\in X_e,\\
  e(t)&=C_ex_e(t) + D_e \distref(t),
  }
\end{subequations}
where $\distref(t)=(\wdist(t),\yref(t))^T$.
If we denote $Q_1 = (I+D\Dc )\inv$ and $Q_2=(I+\Dc D)\inv$, then
$A_e$ and $\Dom(A_e)$ are as in~\eqref{eq:CLsysop}
and
\eq{
B_e &= \pmat{B_d & B\Dc Q_1 \\0&\Bc Q_1}, \quad 
C_e = \pmat{-Q_1\CL& -Q_1D\CcL}, \quad D_e = \pmat{0&Q_1}.
}

The following result shows that the closed-loop system is a regular linear system.  The result also holds whenever $\re\Dc\geq 0$ and $I+D\Dc$ is invertible.

  \begin{lemma}
    \label{lem:CLreg}
    The closed-loop system~\eqref{eq:CLsys} is regular and $A_e$ in~\eqref{eq:CLsysop} generates a contraction semigroup.
  \end{lemma}

\begin{proof}
  Consider the regular linear system
  \eq{
  \left( \pmat{A&0\\0&\Ac }, \pmat{B&B_d&0\\0&0&\Bc }, \pmat{\CL&0\\0&\CcL}, \pmat{D&0&0\\0&0&\Dc } \right).
  }
  The closed-loop system~\eqref{eq:CLsys} is obtained from the above system with output feedback with 
    $\hat{K}=\pmatsmall{0&I\\0&0\\-I&0}$,
  which is an admissible feedback operator since
  $I+D\Dc$ is boundedly invertible by Lemma~\ref{lem:Repostoinv}(d). 
Thus~\eqref{eq:CLsys} is regular~\cite{Wei94}.  
   
  Since $A_e$ generates a semigroup $T_e(t)$ on $X_e$, the Lumer--Phillips Theorem implies that $T_e(t)$ is contactive if $A_e$ is dissipative. The estimates $\re \iprod{Ax+Bu}{x}\leq \re \iprod{\CL x+Du}{u}$ and $\re \iprod{\Ac  z + \Bc y}{z}\leq \re \iprod{\CcL z+\Dc  y}{y}$ and a direct computation show that for any $x_e = (x,z)^T\in \Dom(A_e)$ we have
\eq{
  \re \iprod{A_ex_e}{x_e}
&= \re \iprod{Ax + BQ_2(-\Dc \CL x+\CcL z)}{x} \\
&\quad + \re \iprod{\Ac  z + \Bc  Q_1(-\CL x - D \CcL z)}{z}\\
&\leq 
 \re \iprod{\CL x + DQ_2(-\Dc \CL x+\CcL z)}{Q_2(-\Dc \CL x+\CcL z)} \\
 &\quad + \re \iprod{\CcL  z + \Dc  Q_1(-\CL x - D \CcL z)}{ Q_1(-\CL x - D \CcL z)}\\
 &=0,
}
and thus $A_e$ is dissipative.
\end{proof}

In the following we define the robust output regulation problem for the regular linear system~\eqref{eq:plantfull}. In the problem we consider perturbations for which the perturbed system $(\tilde{A},[\tilde{B},\tilde{B}_d],\tilde{C}_\Lambda,\tilde{D})$ and the perturbed closed-loop system remain regular.
The robustness of the controller also implies that output tracking and disturbance rejection are achieved even if the operators $\Bc$, $\Cc$ and $\Dc$ of the controller are perturbed or approximated in such a way that the closed-loop stability is preserved and the additional conditions on the perturbations stated in Section~\ref{sec:ContrConstructions} are satisfied.

\begin{RORP}
  Choose  $\PARcontr$ in such a way that the following are satisfied:
\begin{itemize}
  \setlength{\itemsep}{.5ex}
  \item[\textup{(a)}] The semigroup $T_e(t)$ generated by
    $A_e$ is strongly stable.
  \item[\textup{(b)}] 
  For the 
  reference and disturbance signals of the form~\eqref{eq:yrefwdist} and for all
  initial states $x_{e0}\in X_e$
  the regulation error satisfies
\eqn{
\label{eq:errintconv}
\int_t^{t+1} \norm{e(s)}ds\to 0 \qquad \mbox{as} \quad t\to \infty.
}
\item[\textup{(c)}] If
  $(A,B,B_d,\CL,D)$ are perturbed to $(\tilde{A},\tilde{B},\tilde{B}_d,\tilde{C}_\Lambda,\tilde{D})$ 
  in such a way that the perturbed closed-loop system is strongly stable, 
  then for the signals~\eqref{eq:yrefwdist} and for all 
initial states $x_{e0}\in X_e$ 
the regulation error satisfies~\eqref{eq:errintconv}.
\end{itemize}
\end{RORP}

  It follows from the results in~\cite[Sec. 3]{Pau17b} that if the closed-loop system is exponentially stable, then
  convergence in~\eqref{eq:errintconv} is uniformly exponentially fast, i.e.,
  there exist $M_e,\ga>0$ such that 
    $\int_t^{t+1} \norm{e(s)} ds\leq M_e e^{-\ga t}(\norm{x_{e0}}+1)$ for all $x_{e0}\in X_e$.
  If the input and output operators of the system and the controller are bounded, then the error convergences pointwise,   
  i.e.,  $\norm{y(t)-\yref(t)}\to 0$ as $t\to \infty$, and the rate is exponential if $T_e(t)$ is exponentially stable.

\section{Passive Controllers for Robust Output Regulation}
\label{sec:ContrConstructions}

The controller constructions in this section are based on the internal model principle~\cite{FraWon75a,PauPoh10,PauPoh14a} which implies that a controller solves the robust output regulation problem provided that its dynamics contain a suitable number of copies of the frequencies $\set{\gw_k}_{k\in\SIndset}$ of the signals~\eqref{eq:yrefwdist} and the closed-loop system is stable.
If $\dim Y<\infty$, then $\PARcontr$ contains an internal model
of the signals~\eqref{eq:yrefwdist} if~\cite[Thm. 13]{Pau17b}
\eq{
  \dim\ker(i\gw_k-\Ac)\geq \dim Y \qquad \forall k\in\SIndset.
}
In the case of an infinite-dimensional output space, the controller contains an internal model if~\cite[Thm. 13]{Pau17b}
\begin{subequations}
    \label{eq:Gconds}
    \eqn{
      \ran(i\gw_k-\Ac )\cap\ran(\Bc )&=\set{0} ~\quad\qquad \forall k\in \SIndset , \label{eq:Gconds1} 
      \\
      \ker(\Bc )&=\set{0}. \label{eq:Gconds2}
    }
  \end{subequations} 

We consider three different situations: 
   In Section~\ref{sec:ContrFinExo} we construct a finite-di\-men\-sional robust controller for a strongly stabilizable system~\eqref{eq:plantfull}. If  $\PARsys$ is exponentially stabilizable, then the convergence of the error is exponentially fast.
   In Section~\ref{sec:Contrtransport} we design a robust controller to track and reject nonsmooth $\tau$-periodic reference signals. The controller is based on a periodic transport equation, and achieves exponential closed-loop stability if the system~\eqref{eq:plantfull} is exponentially stabilizable and satisfies $\re P(i \gw)\geq \eta>0$ for some constant $\eta>0$ near the points $\gw_k=\frac{2\pi k}{\tau}$ for $k\in\Z$.
   In Section~\ref{sec:ContrInfExoDiag} we design an infinite-dimensional robust controller for nonsmooth signals~\eqref{eq:yrefwdist} with a general set of frequencies $\set{\gw_k}_{k\in\SIndset}$. In general, the closed-loop system can not be stabilized exponentially, and we introduce conditions for non-uniform subexponential rates of convergence of the output.

In the constructions we choose the feedthrough of the controller to have the form $\Dc=\Dcone+\Dctwo$, where $\Dctwo\geq 0$ is used to pre-stabilize the system $\PARsys$. We assume that the system $\PARsysS=(A-B\Dctwo Q_1^S \CL,BQ_2^S,Q_1^S\CL,Q_1^SD)$ where
$Q_1^S=(I+D\Dctwo)\inv$ and $Q_2^S=(I+\Dctwo D)\inv$ 
obtained from~\eqref{eq:plantfull} with the output feedback $u(t)=-\Dctwo y(t)$ is either strongly or exponentially stable. Its transfer function is denoted by $P_S(\gl)$.
The passivity of $\PARsys$ implies that also $\PARsysS$ is passive.

\subsection{A Robust Finite-Dimensional Controller}
\label{sec:ContrFinExo}

In this section we assume the signals~\eqref{eq:yrefwdist} contain a finite number of frequencies $\set{\gw_k}_{k=1}^q$, i.e., $\SIndset=\List{q}$.
The controller parameters are chosen in the following way.

\begin{definition}
  \label{def:contrfindim}
  Choose $Z = Y^q$ and
  \eq{
  \Ac  = \diag \left( i\gw_1 I_Y, \ldots, i\gw_q I_Y \right)\in \Lin(Z) ,
  }
  where $I_Y$ is the identity operator on $Y$. Choose $\Cc\in \Lin(Z,Y)$
  of the form
$\Cc z=\sum_{k=1}^q\Cck z_k$ for $z=(z_k)_{k=1}^q\in Z$
so that $\Cck\in \Lin(Y)$
are boundedly invertible for all $k$,
choose  $\Bc = \Cc^\ast $, 
  and choose $\Dc=\Dcone + \Dctwo$ with $\Dcone>0$.
  Finally, choose $\Dctwo\geq 0$ in such a way that $\PARsysS$ is passive and strongly stable with $i\R\subset \rho(A^S)$.
\end{definition}

In the case where $Y$ and $U_d$ are real spaces and $\wdist(\cdot)$ and $\yref(\cdot)$ real-valued functions 
we have
$\set{\gw_k}_{k=1}^q=\set{0,\pm \gw_1,\ldots,\pm \gw_{q'}}$ or $\set{\gw_k}_{k=1}^q=\set{\pm \gw_1,\ldots,\pm \gw_{q'}}$ for some $\gw_1,\ldots,\gw_{q'}>0$. In this case the controller can be chosen to be real by choosing ($J_0$ is omitted if $0\notin \set{\gw_k}_{k=1}^q$)
\eq{
\Ac  = \diag \left(J_0,  J_1,\ldots,J_{q'} \right), \quad 
J_0 = 0\in \Lin(Y), \quad
J_k = \pmat{0&\gw_k I_Y\\-\gw_k I_Y&0}, 
}
and $\Cc = \Cck[0]z_0+\sum_{k=1}^{q'}\Cck z_k^1$ for $z=(z_0,z_1^1,z_1^2,\ldots,z_{q'}^1,z_{q'}^2)\in Z=Y^{2q'+1}$ where $\Cck \in \Lin(Y)$ are boundedly invertible for $0\leq k\leq q'$, $\Bc =\Cc^\ast$, and $\Dc>0$ is as in Definition~\ref{def:contrfindim}.
This controller is passive and it will achieve robust output regulation by Theorem~\ref{thm:ContrConstFin} due to the fact that under the similarity transform
\eq{
V = \diag(I_Y,V_1,\ldots,V_{q'}), \qquad V_k = \frac{1}{\sqrt{2}} \pmat{I_Y&I_Y\\iI_Y&-iI_Y}
}
the system $(V^\ast \Ac V,V^\ast \Bc,\Cc V,\Dc)$ is of the form given in Definition~\ref{def:contrfindim}.

\begin{theorem}
  \label{thm:ContrConstFin}
  The controller in Definition~\textup{\ref{def:contrfindim}} solves the robust output regulation problem. The closed-loop system is strongly stable and $i\R\subset \rho(A_e)$.

  If $\PARsysS$ is exponentially stable, then also the closed-loop system is exponentially stable and 
  for any $\yref(t)$ and $\wdist(t)$ 
  there exist $M_e,\ga>0$ such that 
  \eq{
    \int_t^{t+1} \norm{e(s)} ds\leq M_e e^{-\ga t}(\norm{x_{e0}}+1) \qquad \forall x_{e0}\in X_e.
  }

  In both cases the controller is robust with respect to all perturbations that preserve the stability of the closed-loop system and for which $i\R\subset \rho(\tilde{A}_e)$.
\end{theorem}

\begin{proof}
  The controller $(\Ac ,\Bc,\Cc,\Dcone)$ is passive and its transfer function $\G(\gl)$ satisfies $\re \G(i\gw)=\Dcone>0$ for all $\gw\in\R \setminus \set{\gw_k}_{k=1}^q$.
The operators $(\Ac ,\Bc)$  satisfy~\eqref{eq:Gconds}.
Indeed, the injectivity of $\Bc$ in~\eqref{eq:Gconds2} follows directly from the fact that the components $\Cck^\ast$ of $\Bc$ are boundedly invertible by assumption. Condition~\eqref{eq:Gconds1} can be verified using the diagonal structure of $\Ac $ and the invertibility of $\Cck^\ast$.

 To prove closed-loop stability, we apply Theorem~\ref{thm:CLstabstr} to $\PARsysS$ and $(\Ac,\Bc,\Cc,\Dcone)$.
Condition~(2) of the theorem is satisfied since for any $\gw\in \R\setminus \set{\gw_k}_{k=1}^q$ we have $\re G(i\gw)= \re(\Cc R(i\gw,\Ac )\Bc +\Dcone) = \Dcone>0$, and
condition~(3)
is satisfied by Lemma~\ref{lem:CLstabstrongsuff} since $\Cck$ are invertible.
Thus the strong and exponential closed-loop stabilities follow from Theorem~\ref{thm:CLstabstr}.
Finally, the conclusion that the controller solves the robust output regulation problem follows from~\cite[Thm. 13]{Pau17b}. The results in~\cite{Pau17b} are presented for controllers with $\Dc=0$, but they are applicable since $\Dc\geq 0$ can be written as an output feedback for the system~\eqref{eq:plantfull} without changing the properties of the closed-loop system. 
    Moreover, the results are presented for an infinite set $\set{\gw_k}_{k\in\SIndset}$, but they also apply trivially when $\SIndset$ is finite.
\end{proof}

\begin{proposition}
  \label{prop:ContrFinPointwise}
   The regulation error in Theorem~\textup{\ref{thm:ContrConstFin}} converges pointwise, i.e., $\norm{e(t)}\to 0$ as $t\to \infty$,
for all initial states $x_{e0}\in X_e$ satisfying $A_ex_{e0}+B_e\distref(0)\in X_e$.  
  If the closed-loop system is exponentially stable, 
then for all $\yref(t)$ and $\wdist(t)$ 
  there exist $M_e,\ga>0$ such that
  \eq{
    \norm{e(t)}\leq M_ee^{-\ga t}(\norm{A_ex_{e0}+B_e\distref(0)}+1)
  }
  for all $x_{e0}\in X_e$ satisfying $A_ex_{e0}+B_e\distref(0)\in X_e$.
\end{proposition}

The proof of Proposition~\ref{prop:ContrFinPointwise} is based on the following technical lemma,
which is also used later in the following sections.
The assumptions on $H$ are automatically satisfied if $\SIndset$ is finite, or if the closed-loop system is exponentially stable. In the latter case the property $Hv\in \Dom(\CeL)$ can be verified similarly as in the proof of Theorem~\ref{thm:ContrConstInf}.

\begin{lemma}
  \label{lem:RegErrForm}
  Assume the controller solves the robust output regulation problem
and $\yref(t)$ and $\wdist(t)$ are such that
for some fixed $(f_k)_k\in \lp[2](\C)$ the operator $H: \Dom(H)\subset \lp[2](\C)\to X_e$ defined by
\eq{
  Hv = \sum_{k\in\SIndset} f_k\inv R(i\gw_k,A_e)B_e\distrefk v_k, \qquad  v=(v_k)_k
}
satisfies $H\in \Lin(\lp[2](\C),X_e)$ and $Hv \in\Dom(\CeL)$ for all $v\in \lp[2](\C)$.
If $\yref(t)$ and $\wdist(t)$ are such that the series 
\eqn{
  \label{eq:qextdef}
  \qext = \sum_{k\in \SIndset} i\gw_k R(i\gw_k,A_e)B_e \distrefk.
}
converges in $X_e$, 
then for all $x_{e0}\in X_e$ satisfying $A_e x_{e0}+B_e\distref(0)\in X_e$ and for almost all $t>0$
we have
\eq{
  e(t) = \CeL T_e(t)A_e\inv (A_ex_{e0}+B_e\distref(0)-\qext).
}
\end{lemma}

\begin{proof}
  It follows from the properties of $H$ and the results in~\cite{Pau17b} that for every $x_{e0}\in X_e$ and almost all $t>0$ the regulation error is given by 
  \eq{
    e(t)=\CeL T_e(t)\Bigl(x_{e0}-\sum_{k\in\SIndset} R(i\gw_k,A_e)B_e\distrefk\Bigr).
  }
  If 
  $A_ex_{e0}+B_e\distref(0)\in X_e$, then 
  a direct computation 
  and $\qext\in X_e$
  show
  \eq{
  \MoveEqLeft A_e \sum_{k\in\SIndset} R(i\gw_k,A_e)B_e\distrefk
    = \sum_{k\in\SIndset} i\gw_k R(i\gw_k,A_e)B_e \distrefk -B_e \distref(0),
  }
  which implies the claim.
 \end{proof}

 \begin{proof}[Proof of Proposition~\textup{\ref{prop:ContrFinPointwise}}]
   Since $\SIndset$ is finite, the conditions of Lemma~\ref{lem:RegErrForm} are satisfied. If $x_{e0}\in X_e$ is such that $A_ex_{e0}+B_e\distref(0)\in X_e$, then the estimate
   \ieq{
     \norm{e(t)}\leq \norm{\CeL A_e\inv} \norm{T_e(t)} \norm{A_ex_{e0}+B_e\distref(0)-\qext}
   }
   implies both claims of the proposition.
 \end{proof}

   The following sufficient condition for $A_ex_{e0}+B_e\distref(0)\in X_e$ follows directly from the structures of $A_e$ and $B_e$. 
   Later in Section~\ref{sec:ContrNonuniform} the same condition implies a non-uniform decay rate for the regulation error.

\begin{lemma}
  \label{lem:ContrAeBedomSuff}
  If $\Bc\in \Lin(U,X)$, $\Cc\in \Lin(X,Y)$, and
  $\wdist(0)=0$, then $A_ex_{e0}+B_e\distref(0)\in X_e$ is satisfied for $x_{e0}=(x_0,z_0)^T\in \Dom(A)\times \Dom(\Ac)$ if $\Cc z_0 = \Dc(Cx_0-\yref(0))$.
\end{lemma}

\subsection{A Robust Controller for $\tau$-Periodic Signals}
\label{sec:Contrtransport}

In this section we will construct a regular linear controller that achieves exponentially fast output regulation of $\tau$-periodic reference and disturbance signals.
The controller structure is based on a shift semigroup with periodic boundary conditions, and is related to controllers constructed in~\cite{HarYam88,WeiHaf99,Imm07a}.
 We assume that $\dim Y = p<\infty$, and that $\yref(t)$ and $\wdist(t)$ are $\tau$-periodic functions, i.e., $\SIndset=\Z$ and $\set{\gw_k}_{k\in\Z} = \set{ \frac{2\pi k}{\tau}}_{k\in\Z}$. 

\begin{definition}
  \label{def:contrtransp}
  Choose the controller as
  \begin{subequations}
    \label{eq:contrtransp}
    \eqn{
      \label{eq:contrtransp1}
      z_t(\xi,t) &= z_\xi(\xi,t), \qquad \xi\in(0,\tau), \quad t\geq 0,\\
      z(\cdot,0)&= z_0(\cdot)\in\Lp[2](0,\tau;\C^p),\\
      e(t) &= 2^{-1/2}(z(\tau,t)-z(0,t)),\\
     u(t) &= 2^{-1/2}(z(\tau,t)+z(0,t)) 
      + (\Dconeadd +\Dctwo)e(t)
    }
  \end{subequations}
  where $z(\xi,t)=(z_1(\xi,t),\ldots,z_p(\xi,t))^T$ and $\Dconeadd>0$.
  Choose $\Dctwo\geq 0$ in such a way that $\PARsysS$ is passive and exponentially stable.
\end{definition}

To achieve closed-loop stability, we also assume that $\re P_S(i\gw_k)\geq \eta>0$ for some constant $\eta>0$ and for all $k\in\Z$. If this condition is not satisfied, then exponential closed-loop stability is unachievable, but strong closed-loop stability can be studied using Theorem~\ref{thm:ContrConstInf} in the next section.

\begin{theorem}
  \label{thm:ContrConstTransp}
  Let $\yref(t)$ and $\wdist(t)$ be as in~\eqref{eq:yrefwdist} with $\gw_k = \frac{2\pi k}{\tau}$ for some $\tau>0$. 
  Assume 
  there exist $\eta,\eps>0$ such that $\re P_S(i\gw)\geq \eta>0$ for $\gw\in \Omega_\eps = \setm{\gw\in \R}{\exists k\in\Z:\abs{\gw-\gw_k}<\eps}$, and $\re  D>0$.
   Then the controller in Definition~\textup{\ref{def:contrtransp}} solves the robust output regulation problem in such a way that 
the closed-loop system is exponentially stable, and 
  there exist $M_e,\ga>0$ such that 
  \eq{
    \int_t^{t+1} \norm{e(s)} ds\leq M_e e^{-\ga t}(\norm{x_{e0}}+1) \qquad \forall x_{e0}\in X_e.
  }
The controller is robust with respect to all perturbations that preserve the exponential closed-loop stability, and for which $u(t)=-\Dctwo y(t)$ remains an admissible output feedback and
$\set{i\gw_k}_{k\in\Z}\subset \rho(\tilde{A}^S)$.
\end{theorem}

\begin{proof}
  The controller in Definition~\ref{def:contrtransp} consists of $p=\dim Y$ independent one-dimensional periodic transport equations with boundary control and observation, and an additional feedthrough $(\Dconeadd+\Dctwo)e(t)$.
  The system~\eqref{eq:contrtransp} defines a regular linear system with state $z(t)=z(\cdot,t)$ on
  $Z=\Lp[2](0,\tau;\C^p)$~\cite[Thm. 2.4]{ZwaLeg10},
  and a direct computation shows that its transfer function from $e(t)$ to $u(t)$ is 
  \eq{
    G_0(\gl) = \frac{1+e^{-\gl \tau}}{1-e^{-\gl \tau}}I + \Dconeadd + \Dctwo, \qquad \gl\notin \Bigl\{i \frac{2\pi k}{\tau}\Bigr\}_{k\in\Z}.
  }
  Thus the controller can be written as a system $\PARcontr$ on $Z$ where $A_c$ satisfying 
  $\Ac f  = f'$ for $f\in \Dom(\Ac )=\setm{f\in H^1(0,\tau;\C^p)}{f(0)=f(\tau)}$ generates a unitary group with spectrum $\gs(\Ac )=\set{i \frac{2\pi k}{\tau}}_{k\in\Z}$. 
We also have $\dim \ker(i\gw_k-\Ac )=\dim Y$ for every $k\in\Z$, and thus $\Ac$ contains an internal model of the signals~\eqref{eq:yrefwdist}.
  By~\cite[Thm. 13]{Pau17b} the controller solves the robust output regulation problem if the closed-loop system is exponentially stable.

   To show closed-loop stability, we will verify the conditions of Theorem~\ref{thm:CLstabexp}  for the systems $\PARsysS$ and $\PARcontrone$ 
   with $\Omega=\Omega_\eps$.
   For this we will consider the controller with inputs and outputs
   \eq{
     \uc(t) &= 2^{-1/2}(z(\tau,t)-z(0,t)),\\
     \yc(t) &= 2^{-1/2}(z(\tau,t)+z(0,t) )
     + (\Dconeadd +\Dctwo) \uc(t).
   }
The feedthrough operator of the controller is given by
  \ieq{
    \Dc =  \lim_{\gl\to \infty}G_0(\gl) = I+\Dconeadd + \Dctwo.
  }
  Without the component $(\Dconeadd+\Dctwo)\uc(t)$ of the feedthrough 
  the solutions of~\eqref{eq:contrtransp} satisfy $\ddb{t}\norm{z(t)}_{\Lp[2]}^2 = 2 \re \iprod{u_c(t)}{y_c(t)}$, and thus the controller is passive by~\cite[Thm. 4.2]{Sta02}. 
  Let $\dc>0$ be such that $\Dconeadd\geq \dc>0$. The transfer function $\G(\gl)$ of $(\Ac,\Bc,\CcL,I+\Dconeadd)$ satisfies $\re \G(i\gw) = \Dconeadd\geq \dc>0$ for all $\gw\in\R\setminus \set{\gw_k}_{k\in\Z}$, and thus condition~(2) of Theorem~\ref{thm:CLstabexp} is satisfied.
  To show that condition~(3) of Theorem~\ref{thm:CLstabexp} is satisfied, it is sufficient to show that for any
$D_0\in \Lin(U)$ with $\re D_0>0$
the system $(\Ac,\Bc,\CcL,I+\Dconeadd)$ is stabilized exponentially with feedback $\uc(t)=-D_0 \yc(t)$.
The feedback leads to a partial differential equation 
    \eq{
      z_t(\xi,t) &= z_\xi(\xi,t), \qquad \xi\in(0,\tau), \quad t\geq 0,\\
       (I+\Dtot)z(\tau,t)&= (I-\Dtot)z(0,t). 
    }
where $\Dtot = D_0(I+\Dconeadd D_0)\inv$.
    The exponential stability of this system follows from a straightforward application 
    of~\citel{VilZwa09}{Thm. III.2},
since $\re \Dtot > 0$ by Lemma~\ref{lem:Repostoinv}(c).
Thus Theorem~\ref{thm:CLstabexp} shows that the closed-loop system is exponentially stable.
\end{proof}

  \begin{remark}
    \label{rem:errorconvsize}
    The results in~\cite{Pau17b} also show that if
    $(\yrefk)_k = (a_ky_k)_k $ and $(\wdistk)_k=(a_kw_k)_k$
    where $(y_k)_k\in\lp[2](Y)$, $(w_k)_k\in\lp[2](U_d)$ are fixed, and
    $(a_k)_k\in \lp[2](\C)$,
    then there exist $M_e,\ga >0 $ such that 
\eq{
  \int_t^{t+1} \norm{e(s)} ds\leq M_ee^{-\ga t} \left( \norm{x_{e0}}+ \norm{(a_k)_k}_{\lp[2]}  \right)
}
    for all $ x_{e0}\in X_e$ and $ (a_k)_k\in \lp[2](\C)$.
  \end{remark}

  Lemma~\ref{lem:RegErrForm} implies the following result on the pointwise convergence of $\norm{e(t)}$. The conditions require that
$\yref(t)$ and $\wdist(t)$ have a sufficient levels of smoothness. 
\begin{corollary}
  \label{cor:ContrTranspPointwise}
  If the signals~\eqref{eq:yrefwdist} are such that 
$(k \yrefk)_k\in \lp[1]( Y)$
  and $(k \wdistk)_k\in \lp[1](U_d)$,
  then in Theorem~\textup{\ref{thm:ContrConstTransp}}
  there exist $M_e,\ga>0$ such that
  for all $x_{e0}\in X_e$ satisfying $A_ex_{e0}+B_e\distref(0)\in X_e$ we have
  \eq{
    \norm{e(t)}\leq M_ee^{-\ga t}(\norm{A_ex_{e0}+B_e\distref(0)}+1).
  }
\end{corollary}

If $P(i\mu_j)$ is not invertible for some $\set{i\mu_j}_{j=1}^N\subset \set{i \frac{2\pi k}{\tau}}_{k\in\Z}$, for example for $\mu_j=0$, then the robust output regulation problem is not solvable for signals $\yref(t)$ and $\wdist(t)$ containing these frequencies.  In this situation we can modify the controller in Definition~\ref{def:contrtransp} by replacing~\eqref{eq:contrtransp1} with
\eq{
  z_t(\xi,t) &= z_\xi(\xi,t)-\frac{1}{\tau} \sum_{j=1}^N\sum_{k=1}^p e_k\cdot e^{i\mu_j \xi}\int_0^\tau z_k(s,t)e^{-i\mu_j s}ds, \quad \xi\in(0,\tau) ,
}
where $\set{e_k}_{k=1}^p$ are the Euclidean basis vectors of $\C^p$.
This corresponds to stabilizing the eigenvalues $\set{i\mu_j}_{j=1}^N$ of the transport system~\eqref{eq:contrtransp}, and the resulting controller has the property  $\gs(\Ac )\cap i\R=\set{i \frac{2\pi k}{\tau}}_{k\in\Z}\setminus \set{i\mu_j}_{j=1}^N$.
With this modification the system operator of the controller is of the form $\Ac = \Ac^0-B_0B_0^\ast$ with $B_0\in \Lin(\C^{Np},Z)$. The controller is again passive and is stabilized exponentially with feedback $u_c(t)=-D_0 y_c(t)$ with $\re D_0>0$, and the exponential closed-loop stability follows from Theorem~\ref{thm:CLstabexp}.

\subsection{A Robust Controller for Nonsmooth Signals}
\label{sec:ContrInfExoDiag}

In this section we construct an infinite-dimensional diagonal controller for signals~\eqref{eq:yrefwdist} with a general set $\set{\gw_k}_{k\in\Z}$ of distinct frequencies with no finite accumulation points.
The controller can also be used for systems with an infinite-dimensional output space $Y$.
If $\yref(t)$ and $\wdist(t)$ are $\tau$-periodic and $\dim Y<\infty$, then the controller is of similar form as in Definition~\ref{def:contrtransp}.

\begin{definition}
  \label{def:contrinfdim}
Choose  $Z=\lp[2](\SIndset;Y)$ and 
  \eq{
  \Ac  
  &= \diag(i\gw_k I_Y)_{k\in\SIndset}, \quad \Dom(\Ac ) = \Setm{(z_k)_k\in Z}{(\abs{\gw_k}\norm{z_k})_k\in \lp[2](\C)},
  }
  where $I_Y$ is the identity operator on $Y$. Let $\Dc=\Dcone + \Dctwo$ with $\Dcone>0$ and $\Dctwo\geq 0$. 
    Choose admissible  $\Bc\in \Lin(Y,Z_{-1})$ and $\Cc\in \Lin(Z_1,Y)$ as
    \eq{
      \Bc y = (\Bck y)_k \quad \forall y\in Y,
      \qquad 
      \Cc z &= \sum_{k\in\SIndset}\Bck^\ast z_k \quad\; \forall z\in \Dom(\Ac ),
    }
    with boundedly invertible $\Bck\in \Lin(Y)$ 
    so that $\PARcontrone$ is a regular linear system
    whose transfer function $\G(\gl)$  satisfies $\re G(i\gw)\geq \dc>0$ for some constant $\dc>0$ and for all $\gw\in\R\setminus \set{\gw_k}_{k\in\SIndset}$.  
Finally, choose $\Dctwo\geq 0$ in such a way that $\PARsysS$ is passive and strongly stable with $i\R\subset \rho(A^S)$.  
\end{definition}

If $\dim Y<\infty$ and $\set{\gw_k}_{k\in\SIndset}$ has a uniform gap, i.e., $\inf_{k\neq l}\abs{\gw_k-\gw_l}>0$, then~\cite[Cor. 5.2.5, Prop. 5.3.5]{TucWei09book} imply that $\Bc$ and $\Cc$ are admissible with respect to $\Ac$ if $(\norm{\Bck})_{k\in\SIndset}\in \lp[\infty](\C)$ and $(\norm{\Cck})_{k\in\SIndset}\in \lp[\infty](\C)$. For more general conditions for admissibility, see~\cite[Sec. 5.3]{TucWei09book}.
The system $\PARcontrone$ is regular whenever $\Bc$ and $\Cc$ are admissible and there exists $\eps>0$ such that $((1+\abs{\gw_k})^{-1/2+\eps}\norm{\Bck})_k\in \lp[2](\C)$~\cite[Prop. 4.1]{CurWei06}.
However, there are also regular linear systems, such as the controller in Definition~\ref{def:contrtransp}, for which neither of these conditions is satisfied.
  If $\set{\gw_k}_{k\in\Z}$ has a uniform gap, $(\abs{\gw_k}^{\eps}\norm{\Bck})_k\in \lp[\infty](\C)$ for some $\eps>0$ and $\Dcone>0$, 
  then $\PARcontrone$ satisfies the conditions of Definition~\ref{def:contrinfdim}.

Due to the lack of exponential closed-loop stability, 
the solvability of the robust output regulation problem requires additional conditions on the reference and disturbance signals. These conditions  relate the behaviour of the
coefficients $\yrefk$ and $\wdistk$
to the behaviour of the transfer functions
$P(\gl)$ and $P_d(\gl)$
on the frequencies $\set{\gw_k}_{k\in\SIndset}$.
We pose conditions on the sequences $\Pitrans = (\Pitransk)_{k\in\SIndset}\subset \XBBd\times Y$ consisting of the elements $\Pitransk=(\Pitranskone,\Pitransktwo)$ with
\eq{
\Pitranskone &=R(i\gw_k,A^S) B^S \Pipartk +  R(i\gw_k,A^S) B_d\wdistk, \\
\Pitransktwo &=(\Bck^\ast)\inv (\Pipartk - \Dctwo \yrefk),
}
where
$\Pipartk = P_S(i\gw_k)\inv \yrefk- P_S(i\gw_k)\inv \CL^S R(i\gw_k,A^S) B_d\wdistk $.
In the case of a perturbed system, we define $\Pitranspert=(\Pitranskpert)_{k\in\SIndset}$ analogously.
Alternate ways of expressing $\Pitransk$ are presented in Lemma~\ref{lem:Pitransalt}.
Note in particular that if $\PARsysS$ is exponentially stable, then~\eqref{eq:RORPEFcondShapen} are satisfied provided that
$(\norm{u_k})_k\in\lp[1](\C)$ and 
$(\norm{\Bck\inv}\norm{\Pipartk-\Dctwo \yrefk})_k\in\lp[2](\C)$.

\begin{theorem}
  \label{thm:ContrConstInf}
  Assume $\re P_S(i\gw_k)>0$ for all $k\in\SIndset$.
  The controller
in Definition~\textup{\ref{def:contrinfdim}}
solves the robust output regulation problem for all 
$\yref(t)$ and $\wdist(t)$ whose coefficients satisfy
  \eqn{
    \label{eq:RORPEFcondShapen}
    (\Pitranskone)_k\in \lp[1](X), \quad\ 
    (\Pitransktwo)_k\in \lp[2](Y), \quad\ 
    (\Pipartk)_k\in \lp[1](U).
  }
The closed-loop system is strongly stable and $i\R\subset \rho(A_e)$.

The controller is robust with respect to all perturbations $\PARsysopspert$ for which
$u(t)=-\Dctwo y(t)$ remains an admissible output feedback, 
the strong closed-loop stability is preserved, $\set{i\gw_k}_{k\in\SIndset}\subset\rho(\tilde{A}_e)\cap \rho(\tilde{A}^S)$,
 $\tilde{P}_S(i\gw_k)$ are invertible for $k\in \SIndset$,
and 
$ (\Pitranskpert )_{k\in\SIndset}$ satisfies~\eqref{eq:RORPEFcondShapen}.

If the closed-loop system is exponentially stable, then~\eqref{eq:RORPEFcondShapen} are satisfied automatically, and
  there exist $M_e,\ga>0$ such that 
  \ieq{
    \int_t^{t+1} \norm{e(s)} ds\leq M_e e^{-\ga t}(\norm{x_{e0}}+1) 
  }
  for all $x_{e0}\in X_e$.
\end{theorem}

\begin{proof}
  The proof is based on the application of~\cite[Thm. 13]{Pau17b}.
  The diagonal structure of the controller and the invertibility of $\Bck$ imply that $\Ac$ and $\Bc$ satisfy the conditions~\eqref{eq:Gconds}.
To show that the closed-loop system is strongly stable, we apply Theorem~\ref{thm:CLstabstr}
 for the systems $\PARsysS$ and $\PARcontrone$. 
 Conditions~(1) and~(2) are satisfied due to the construction in Definition~\ref{def:contrinfdim}, and
condition~(3) 
is satisfied
by Lemma~\ref{lem:CLstabstrongsuff} since $\Cck=\Bck^\ast$ are invertible.
Thus by Theorem~\ref{thm:CLstabstr} the closed-loop system is strongly stable and $i\R\subset \rho(A_e)$.

To apply~\cite[Thm. 13]{Pau17b} directly, we would need  
  $R(i\gw_k,A_e)B_e \distrefk
  \in \lp[1](X_e)$. However, in~\cite{Pau17b} this property is used as a sufficient condition for the existence of $(f_k)_k\in\lp[2](\C)$ such that the operator $H: \Dom(H)\subset \lp[2](\C)\to X_e$
  in Lemma~\ref{lem:RegErrForm}
  satisfies $H\in \Lin(\lp[2](\C),X_e)$ and $\ran(H)\subset \Dom(\CeL)$. 
  Here we will verify that the sequence 
$(f_k)_k\in\lp[2](\C)$ 
 with
 \eq{
   f_k = 
   \begin{cases}
     \norm{\Pitransktwo}+(\norm{\distrefk}+\norm{\Pitranskone}+ \norm{u_k})^{1/2} &\mbox{if}~\distrefk\neq 0 \\
     2^{-\abs{k}}&\mbox{if}~\distrefk=0
   \end{cases}
 }
 has this property.
If $k\in\SIndset$ and $x_e^k = (\Pitranskone,z_k)\in \XBBd\times \ZG$ where
\eq{
  z_k &= (z_k^j)_{j\in\SIndset}, \quad 
  z_k^k = \Pitransktwo , \quad z_k^j=0, ~ j\neq k,
  } 
then it is straightforward to verify that $(i\gw_k-A_e)x_e^k = B_e\distrefk$, and thus we have $R(i\gw_k,A_e)B_e\distrefk = (\Pitranskone,z_k)$.
  Now $(f_k\inv (\norm{\distrefk}+\norm{\Pitranskone}+\norm{u_k}))_k\in\lp[2](\C)$ and $(f_k\inv \Pitransktwo)_k\in\lp[\infty](Y)$.
  These properties and the structure of $R(i\gw_k,A_e)B_e\distrefk$ imply that $Hv$ is well-defined for every $v\in \lp[2](\C)$, and 
  \eq{
    \norm{Hv}^2 &= 
\biggl\|\sum_{k\in\SIndset} f_k\inv\Pitranskone  v_k\biggr\|_X^2 + \bigl\| \left( f_k\inv\Pitransktwo  v_k \right)_k\bigr\|_{\lp[2](Y)}^2\\
&\leq \norm{v}^2 \norm{(f_k\inv\Pitranskone )_k}_{\lp[2](X)}^2 +
            \norm{v}^2  \norm{(f_k\inv\Pitransktwo )_k}_{\lp[\infty](Y)}^2
  }
  implies $H\in \Lin(\lp[2](\C),X_e)$.
  It remains to show $\ran(\Sigma )\subset \Dom(\CeL)$. 
 If we denote $\Pez(\gl)=\CeL R(\gl,A_e)B_e$, 
 then
 \ieq{
   \Pez(i\gw_k) \distrefk 
   = -Q_1(\CL \Pitranskone + D(\Pipartk-\Dctwo\yrefk))
 }
for every $k\in\SIndset$.
 The regularity of $\PARsysS$ and~\eqref{eq:RORPEFcondShapen} imply $(f_k\inv \Pez(i\gw_k)\distrefk)_k\in \lp[2](Y)$.
 If $v\in \lp[2](\C)$ and $\gl>0$,  
  the resolvent identity implies
  \eq{
   \gl \CeL R(\gl,A_e)Hv
   &=\sum_{k\in\SIndset} \frac{\gl f_k\inv v_k}{\gl-i\gw_k} \Pez(i\gw_k) \distrefk 
- \Pez(\gl)\sum_{k\in\SIndset} \frac{\gl f_k\inv v_k}{\gl-i\gw_k}  \distrefk \\
&\quad \longrightarrow
   \sum_{k\in\SIndset} f_k\inv \Pez(i\gw_k) \distrefk v_k
  }
  as $\gl\to \infty$
  since $(A_e,B_e,C_e)$ is regular and since $(f_k\inv \Pez(i\gw_k)\distrefk v_k)_k\in \lp[1](Y)$ and $(f_k\inv \distrefk v_k)\in \lp[1](U_d\times Y)$.
  Thus $Hv\in \Dom(\CeL)$ by definition.  
An analogous argument shows that for perturbed systems $\PARsysopspert$ the sequence $(f_k)_k$ 
can again be chosen so that $\tilde{H}$ has the required properties. 
Thus the claims of the theorem follow from~\cite[Thm. 13]{Pau17b}.
If the closed-loop system is exponentially stable, then $(\Pitranskone,z_k)=R(i\gw_k,A_e)B_e\distrefk$ implies 
$(\Pitransk)_k\in \lp[1](X\times Y)$,
which also shows $(\norm{u_k})_k\in \lp[1](\C)$.
\end{proof}

The following alternate expressions for $\Pitransk$ can be verified using standard operator identities and Lemma~\ref{lem:Woodbury}.

\begin{lemma}
\label{lem:Pitransalt}
    If $i\gw_k\in\rho(A)$ 
      for some $k\in \SIndset$, then 
      \eq{
\Pitranskone &=  R(i\gw_k,A)B_d\wdistk+R(i\gw_k,A)B \Pipartkalt \\
  \Pitransktwo &= (\Bck^\ast)\inv \Pipartkalt, \qquad \Pipartk = \Pipartkalt + \Dctwo \yrefk
      }
where
$\Pipartkalt = P(i\gw_k)\inv \yrefk- P(i\gw_k)\inv P_d(i\gw_k)\wdistk $.
If $D$ is boundedly invertible, then
\ieq{
 \Pitranskone= R_k^DB_d \wdistk +  R(i\gw_k,A^S)B^S P_S(i\gw_k)\inv \yrefk
 }
for all $k\in\SIndset$,
 where $R_k^D = R(i\gw_k,A^S-B^S(D^S)\inv \CL^S)$.
\end{lemma}

The following result shows that pointwise convergence is achieved for sufficiently smooth signals $\yref(t)$ and $\wdist(t)$ and for suitable intial states. 

\begin{proposition}
  \label{prop:ContrInfPointwise}
  Assume $\yref(t)$ and $\wdist(t)$
  are such that
    $( \gw_k\Pitranskone )_k\in \lp[1](X)$ and
    $(\gw_k \Pitransktwo )_k\in \lp[2](Y)$. 
    If 
 $x_{e0}\in X_e$ and $A_ex_{e0}+B_e\distref(0)\in X_e$, then the regulation error in Theorem~\textup{\ref{thm:ContrConstInf}} satisfies 
$\norm{e(t)}\to 0$ as $t\to \infty$.  
  If the closed-loop system is exponentially stable, 
  then
there exist $M_e,\ga>0$ 
such that
    \eq{
    \norm{e(t)}\leq M_ee^{-\ga t}(\norm{A_ex_{e0}+B_e\distref(0)}+1)
  }
for all $x_{e0}\in X_e$ satisfying $A_ex_{e0}+B_e\distref(0)\in X_e$.  
\end{proposition}

\begin{proof}
  As in the proof of Theorem~\ref{thm:ContrConstInf},  $R(i\gw_k,A_e)B_e \distrefk =(\Pitranskone,z_k)$ where $z_k=(z_k^j)_j$ is such that $z_k^k=\Pitransktwo$ and $z_k^j=0$ for $j\neq k$. 
  This structure, 
    $( \gw_k\Pitranskone )_k\in \lp[1](X)$, and
    $(\gw_k \Pitransktwo )_k\in \lp[2](Y)$ imply
    that $\qext$ in~\eqref{eq:qextdef} satisfies $\qext\in X_e$.
  Since the required properties of $H$ were verified in
  the proof of Theorem~\ref{thm:ContrConstInf}, 
   the claims follow from Lemma~\ref{lem:RegErrForm}.
\end{proof}

\subsection{Non-Uniform Convergence Rates of the Regulation Error}
\label{sec:ContrNonuniform}

We will now use Theorem~\ref{thm:CLstabnonuniform} to derive convergence rates for the regulation error
in Theorem~\ref{thm:ContrConstInf}. The estimates are valid for reference and disturbance signals with sufficient levels of smoothness. In particular, we assume $\set{\gw_k}_{k\in\SIndset}$ has a uniform gap and
the coefficients of $\yref(t)$ and $\wdist(t)$  satisfy
  \eqn{
  \label{eq:RORPnonunifEFcond}
    \left( \gw_k\Pitranskone \right)_{k\in\SIndset}\in \lp[1](X), \qquad
    \left( \gw_k\Pitransktwo \right)_{k\in\SIndset}\in \lp[2](Y),
  } 
  which is a strictly stronger condition than the first two parts of~\eqref{eq:RORPEFcondShapen}.

\begin{theorem}
  \label{thm:ContrConstNonunif}
  Assume
  $\PARsysS$ is passive and exponentially stable, the controller is as in Definition~\textup{\ref{def:contrinfdim}},
  and the conditions of Theorem~\textup{\ref{thm:ContrConstInf}} are satisfied.

Assume there exists $0<\eps<\frac{1}{2}\inf_{k\neq l}\abs{\gw_k-\gw_l}$ 
such that $\re P_S(i\gw)>0$ for all $\gw \in \Omega_\eps = \setm{\gw\in \R}{\exists k\in\SIndset :\abs{\gw-\gw_k}<\eps}$.
Let $\eta(\cdot),\gg(\cdot):\R_+\to (0,1]$ be continuous  non-increasing 
functions 
with the property $\inf_{\gw>0}\gg(\gw+\gd_0)/\gg(\gw)>0$ for some  $0<\gd_0<\min \set{1,\eps}$
such that the following hold.
\begin{itemize}
  \item $\re P_S(i\gw)\geq \eta(\abs{\gw})$
      for all $\gw\in \Omega_\eps$.
  \item $\norm{\Bck^\ast y}\geq \gg(\abs{\gw_k}) \norm{y}$ for all $k\in\SIndset$ and $y\in Y$.
\end{itemize}
    Then the controller solves the robust output regulation problem and there exists 
    $M_0>0$ such that $\norm{R(i\gw,A_e)}\leq M_R(\abs{\gw})$  with $M_R(\cdot)=M_0\eta(\cdot)\inv\gg(\cdot)^{-2}$.  
    If $\sup_{\gw>0} M_R(\gw)<\infty$, then the closed-loop system is exponentially stable. More generally, there exist $M_e^e,t_0\geq 1$ 
 such that
 if~\eqref{eq:RORPnonunifEFcond} hold, then
 for all 
 $x_{e0}\in X_e$ satisfying $A_ex_{e0}+B_e\distref(0)\in X_e$ we have
\eqn{
\label{eq:RORPnonuniformerror}
\int_t^{t+1} \norm{e(s)}ds\leq \frac{M_e^e}{M_T(t)} 
\left( \norm{A_ex_{e0}+B_e\distref(0)} 
+ \Piseqnorm \right), \qquad t\geq t_0,
}
where $M_T(t)$ is determined by parts \textup{(b)--(c)} of Theorem~\textup{\ref{thm:CLstabnonuniform}}
and
$\Piseqnorm^2=\norm{(\gw_k \Pitranskone)}_{\lp[1]}^2 + \norm{(\gw_k\Pitransktwo)_k}_{\lp[2]}^2$.
In particular, if $\eta(\gw)\inv\gg(\gw)^{-2}=O(\gw^\ga)$ for some $\ga>0$, then~\eqref{eq:RORPnonuniformerror}
holds with $M_T(t)=t^{1/\ga}$.
\end{theorem}

\begin{proof}
Theorem~\ref{thm:ContrConstInf} shows that the controller solves the robust output regulation problem, and
$\norm{R(i\gw,A_e)}\leq M_R(\abs{\gw})$ follows from Theorem~\ref{thm:CLstabnonuniform} and Remark~\ref{rem:CLNUunifgap}. Thus~\eqref{eq:RORPnonuniformstate} holds $M_T(\cdot)$ and for some $M_e,t_0>0$. 
As shown in the proofs of Theorem~\ref{thm:ContrConstInf} and Lemma~\ref{prop:ContrInfPointwise}, the conditions of Lemma~\ref{lem:RegErrForm} are satisfied whenever $\yref(t)$ and $\wdist(t)$ are such that~\eqref{eq:RORPEFcondShapen} and~\eqref{eq:RORPnonunifEFcond} hold. 
If $x_{e0}\in X_e$ is such that $A_ex_{e0}+B_e\distref(0)\in X_e$, then 
\ieq{
  e(t) = \CeL T_e(t)A_e\inv (A_ex_{e0}+B_e\distref(0)-\qext).
}
The admissibility of $\CeL$ and~\eqref{eq:RORPnonuniformstate} imply 
\eq{
  \int_t^{t+1} \norm{e(s)}ds
  &\lesssim \norm{T_e(t)A_e\inv (A_ex_{e0}+B_e\distref(0)-\qext)} \\
  &\leq 
   \frac{M_e^e}{M_T(t)} 
 \left( \norm{ A_ex_{e0}+B_e\distref(0)}+ \norm{\qext} \right),
}
which implies the claim since 
$\norm{\qext}^2\leq \Piseqnorm^2$.
\end{proof}

  If $C\in \Lin(X,Y)$ and $\Cc\in \Lin(Z,U)$ in Theorem~\ref{thm:ContrConstNonunif}, then~\eqref{eq:RORPnonuniformerror} can be replaced with a pointwise rate 
  \ieq{
    \norm{e(t)}\leq 
      \frac{M_e^e}{M_T(t)} 
    \left( \norm{A_ex_{e0}+B_e\distref(0)} 
      + \Piseqnorm \right)
  }
  for $t\geq t_0$.
  If $\wdist(0)=0$ and $B_c\in \Lin(Z,U)$, then Lemma~\ref{lem:ContrAeBedomSuff} gives a sufficient condition for
  initial states $z_0\in Z$ that achieve the convergence rate~\eqref{eq:RORPnonuniformerror}.

The following result presents necessary conditions for exponential closed-loop stability
with controllers satisfying the conditions~\eqref{eq:Gconds}, which in turn are necessary for robustness by~\citel{Pau17b}{Thm. 13}.

\begin{proposition}
  \label{prop:CLexpnecess}
  Assume $\PARsysS$ is strongly stable, $\set{i\gw_k}_{k\in\SIndset}\subset \rho(A^S)$, and
  $\PARcontr$ 
  satisfies~\eqref{eq:Gconds}.
  If the closed-loop system is exponentially stable, then
  $\sup_{k\in\SIndset}\norm{P_S(i\gw_k)\inv}<\infty$.
\end{proposition}

\begin{proof}
  It follows from the proof of
   Lemma~\ref{lem:CLreg} that $B_e^0=\left[ 0\atop \Bc \right]$ and $C_e^0=[0,\CcL]$ are admissible with respect to $A_e$. 
  The proof of Theorem~\ref{thm:CLstabstr} implies
  $C_e^0 R(i\gw_k,A_e) B_e^0 = \CcL S_A(i\gw_k)\inv \Bc$ where $S_A(i\gw_k)=i\gw_k-\Ac+\Bc P_{cl}(i\gw_k)\CcL$ and $P_{cl}(i\gw_k)=P_S(i\gw_k)(I+\Dcone P_S(i\gw_k))\inv $.  
  Since the closed-loop system is exponentially stable, we must have
  \eqn{
    \label{eq:KSAeGfin}
    \sup_{k\in\SIndset}\;\norm{\CcL S_A(i\gw_k)\inv \Bc} < \infty.
  }
  Let $y\in Y$ and denote $z=S_A(i\gw_k)\inv \Bc y\in \ZG$, which implies $(i\gw_k-\Ac)z  =\Bc (y-P_{cl}(i\gw_k)\CcL z)$.
  The conditions~\eqref{eq:Gconds} show that we must have
  $y=P_{cl}(i\gw_k)\CcL z$.
  Thus $\CcL S_A(i\gw_k)\inv \Bc y = P_{cl}(i\gw_k)\inv y = (P_S(i\gw_k)\inv + \Dcone)y$ for all $y\in Y$, and the claim follows from~\eqref{eq:KSAeGfin}.
\end{proof}

\section{Examples}
\label{sec:examples}

\subsection{A Wave Equation with Boundary Control}

We consider a one-dimen\-sion\-al undamped wave equation with boundary control and observation,
  \begin{subequations}
    \label{eq:waveex}
    \eqn{
      w_{tt}(\xi,t)&=w_{\xi\xi}(\xi,t), \qquad \xi\in(0,1)\\
      w_\xi(\xi,0)&=w_0(\xi), \quad w_t(\xi,0)=w_1(\xi),\\
      u(t)&=-w_\xi(0,t), \quad w_\xi(1,t)=0,\\
      y(t)&= w_t(0,t).
    }
  \end{subequations}
  The results in~\cite{ZwaLeg10} show that~\eqref{eq:waveex} defines a regular linear system with state $x(t)=(w_\xi(\cdot,t),w_t(\cdot,t))^T$ on $X=\Lp[2](0,1)\times \Lp[2](0,1)$. Its transfer function is given by
  \eq{
    P(\gl) = \frac{1+e^{-2\gl}}{1-e^{-2\gl}}, \qquad \gl\neq i\pi k, \qquad k\in\Z
  }
  and $D=1$. In particular, we have $\re P(\gl)\geq 0$ for all $\gl\in\C_+$. 
  We will construct a controller that achieves exponential closed-loop stability and robust output regulation for 1-periodic signals of the form
  \ieq{
    \yref(t) = \sum_{k\in\Z} \yrefk e^{i 2\pi k t}
  }
  with $(\yrefk)_k\in\lp[1](\C)$.
  For this we will use a controller based on the transport equation presented in Section~\ref{sec:Contrtransport} with $\tau = 1$.

  The system~\eqref{eq:waveex} can be stabilized exponentially with negative output feedback $u(t)=-\Dctwo y(t)$ with $\Dctwo>0$. For $\gl\in \C_+$ the transfer function $P_S(\gl)$ of the stabilized system $\PARsysS$ is given by
  \eq{
    P_S(\gl) = P(\gl)(I+\Dctwo P(\gl))\inv 
    =\frac{1+e^{-2\gl}}{1+\Dctwo + (\Dctwo-1)e^{-2\gl}}
  }
  and $\re P_S(i\gw)=\frac{\Dctwo \cos(\gw)^2}{1+(\Dctwo^2-1)\cos(\gw)^2}$. 
  Now $\re P_S(i\gw)=0$ if and only if $\gw = (k+1/2)\pi$ for some $k\in\Z$. 
  Therefore for any fixed $0<\eps<\pi/2$ there exists $\eta>0$ such that $\re P_S(i\gw)\geq \eta>0$ for all $\gw\in\Omega_\eps= 
  \setm{\gw\in \R}{\exists k\in\SIndset :\abs{\gw-2\pi k}<\eps}$.

  The conditions of Theorem~\ref{thm:ContrConstTransp} are satisfied, and thus the controller in Definition~\ref{def:contrtransp} solves the robust output regulation problem for all 1-periodic reference signals with $(\yrefk)_k\in\lp[1](\C)$ and the output of the controlled system converges to $\yref(t)$ at an exponential rate.
  The closed-loop system consisting of~\eqref{eq:waveex} and the controller (without the reference signal) becomes
  \eq{
    w_{tt}(\xi,t)&=w_{\xi\xi}(\xi,t), \qquad \xi\in(0,1)\\
    z_t(\xi,t) &= z_\xi(\xi,t), \qquad \xi\in(0,1)\\
    w_\xi(\xi,0)&=w_0(\xi), \quad w_t(\xi,0)=w_1(\xi), \quad z(\xi,t)=z_0(\xi)\\
    w_\xi(0,t) &= (\gb -2^{-1/2})z(0,t) - (\gb+2^{-1/2})z(1,t)  \\
    w_t(0,t) &= 2^{-1/2}(z(0,t)-z(1,t)), \quad w_\xi(1,t)=0
  }
  where $\gb=\Dconeadd + \Dctwo>0$ is arbitrary. By Theorem~\ref{thm:ContrConstTransp} the semigroup $T_e(t)$ associated to this coupled system of partial differential equations is exponentially stable,
    and thus
    $\norm{w_\xi(\cdot,t)}_{\Lp[2]}^2 + \norm{w_t(\cdot,t)}_{\Lp[2]}^2 + \norm{z(\cdot,t)}_{\Lp[2]}^2\to 0$ 
    at an exponential rate as $t\to \infty$.

\subsection{A Strongly Stabilizable Wave Equation}

  In this example we consider another one-dimensional wave equation, now with distributed control and observation,
  \begin{subequations}
    \label{eq:wavestrex}
    \eqn{
      w_{tt}(\xi,t)&=w_{\xi\xi}(\xi,t) + b(\xi)u(t), \qquad \xi\in(0,1)\\
      w(0,t)&=0, \quad w(1,t)=0,\\
      w(\xi,0)&=w_0(\xi), \quad w_t(\xi,0)=w_1(\xi),\\
      y(t)&= \int_0^1 b(\xi)w_t(\xi,t)d\xi,
    }
  \end{subequations}
  where $b(\xi)=2(1-\xi)$. Equation~\eqref{eq:wavestrex} determines a passive linear system
  with state $x(t)=(w(\cdot,t),w_t(\cdot,t))^T$ on $X=H_0^1(0,1)\cap \Lp[2](0,1)$
  with bounded input and output operators satisfying $C=B^\ast$. The transfer function $P(\gl)$ can be computed as in~\cite[Sec. II]{CurMor09}. Negative output feedback $u(t)=-\Dctwo y(t)$ stabilizes the system strongly for any $\Dctwo>0$, but 
  the system is not exponentially stabilizable.
However, the semigroup generated by $A^S$ 
is polynomially stable
since $\int_0^1 b(\xi)\sin(k\pi \xi)d\xi=\frac{2}{k\pi}$ implies $\norm{R(i\gw,A-B\Dctwo C)}=O(\gw^2)$ for $\Dctwo >0$ by~\citel{Rus69}{Thm. 1}.

  Our aim is to design a controller to achieve robust output tracking of 
  \ieq{
    \yref(t) = \sin(\pi t)+\frac{1}{4}\cos(2\pi t).
  }
  The frequencies of the signal $\yref(t)$ are $\set{\pm \pi,\pm 2\pi}$. Due to robustness, the controller will be able to track any reference signal with these frequencies.
  Since $\dim Y=p=1$, we can construct a passive feedback controller in Definition~\ref{def:contrfindim} on $Z=\R^4$ by choosing
  \eq{
    \Ac  = \blkdiag(J_1,J_2), \quad J_1 = \pmat{0&\pi\\-\pi&0}, \quad J_2=\pmat{0&2\pi\\-2\pi&0}, 
  }
  $\Cc =[k_1,0,k_2,0]$, $\Bc=\Cc^\ast$, and $\Dc >0$. 
  The values of $k_1,k_2\in \R$ and $\Dc$ affect the stability properties of the closed-loop system.
  In this example we choose  $k_1=k_2 = 3$ and $\Dc=35$. 
  By construction the controller is robust with respect to perturbations in the system provided that the strong stability of the closed-loop is preserved.
  Since $B$ and $C$ are bounded operators, Proposition~\ref{prop:ContrFinPointwise} shows that $\norm{e(t)}\to 0$ as $t\to \infty$ for all initial states $x_0\in \Dom(A)$ and $z_0\in Z$.

  For simulations, the system~\eqref{eq:wavestrex} was approximated with the Finite Element Method with $N=24$ points on $[0,1]$. Figure~\ref{fig:wave1DexSim} depicts the behaviour of the error $e(t)$ and the integrals $\int_t^{t+1}\norm{e(s)}ds$ for $0\leq t\leq 24$
  for initial states $x_0(\xi)=\xi(1-\xi)(2-5\xi)$ and $z_0=0$. 
   Figure~\ref{fig:wave1DexSim} also plots the solution $w(\xi,t)$ of the controlled wave equation for $0\leq t\leq 6$.

    \begin{figure}[h!]
  \begin{minipage}{.47\linewidth}
      \begin{center}
	\includegraphics[width=1.1\linewidth]{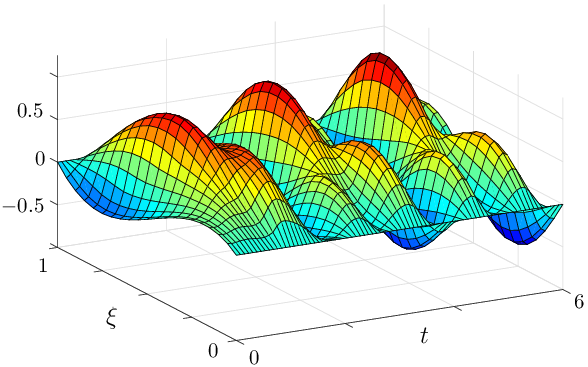}
      \end{center}
  \end{minipage}
  \hfill
  \begin{minipage}{.51\linewidth}
      \begin{flushright}
	\vspace{-1.5ex}

	\includegraphics[width=.9\linewidth]{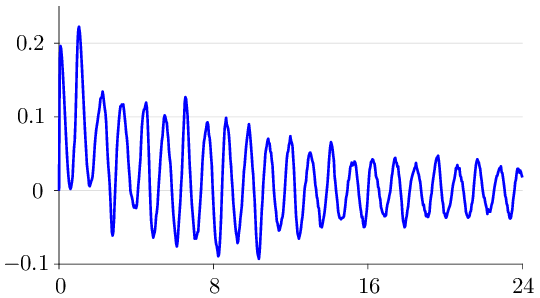}

	\vspace{1ex}

	\includegraphics[width=.878\linewidth]{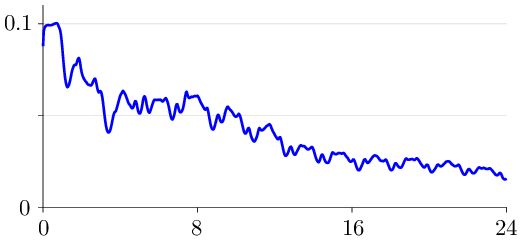}

	\vspace{-1ex}
      \end{flushright}
  \end{minipage}
      \caption{The solution $w(\xi,t)$ of controlled wave equation (left) and $e(t)$ (top right) and $\int_t^{t+1}\norm{e(s)}ds$ (bottom right).}
\label{fig:wave1DexSim}
\end{figure}

\subsection{Periodic Output Tracking for a Heat Equation}

In the final example we consider a two-dimensional boundary controlled heat equation on $\Omega = [0,1]\times [0,1]$
  \begin{subequations}
    \label{eq:heatex}
    \eqn{
      x_t(\xi,t) &= \Delta x(\xi,t), \qquad x(\xi,0)=x_0(\xi) \\
      \pd{x}{n}(\xi,t)\vert_{\Gamma_1} &= u(t), \qquad 
      \pd{x}{n}(\xi,t)\vert_{\Gamma_2} = \wdist(t), \qquad
      \pd{x}{n}(\xi,t)\vert_{\Gamma_0} = 0\\
      y(t) &= \int_{\Gamma_1}x(\xi,t)d\xi,
    }
  \end{subequations}
  where the parts $\Gamma_0$, $\Gamma_1$, and $\Gamma_2$ of the boundary $\partial \Omega$ are defined so that
  $\Gamma_1 = \setm{\xi=(0,\xi_2)}{0\leq \xi_2\leq 1}$,
  $\Gamma_2 = \setm{\xi=(\xi_1,1)}{1/2\leq \xi_1\leq 1}$, 
  $\Gamma_0 = \partial \Omega \setminus (\Gamma_1 \cup \Gamma_2)$. By~\cite[Cor. 2]{ByrGil02} the heat equation defines a regular linear system with state $x(t)=x(\cdot,t)$ on $X=\Lp[2](\Omega)$
  with feedthrough $D=0$. The system is passive,
  \eq{
    P(\gl) = \frac{\coth(\sqrt{\gl})}{\sqrt{\gl}}, \qquad \gl\in \overline{\C_+}\setminus \set{0},
  }
and $\abs{P(i\gw)\inv}=O(\abs{\sqrt{\gw}})$ for $\gw\in\R$ with large $\abs{\gw}$. The system~\eqref{eq:heatex} is exponentially stabilizable with feedback $u(t) = -\Dctwo y(t)$ for any $\Dctwo>0$.

  We will design an infinite-dimensional dynamic feedback controller that achieves robust output tracking of the $2$-periodic nonsmooth reference signal $\yref(t)$ in
  Figure~\ref{fig:heat2DexSim}
  and rejects a suitable class of $2$-periodic disturbance signals $\wdist(t)$. 
  The frequencies of the signals are $\set{\gw_k}_{k\in\Z}$ with $\gw_k = \pi k$ for $k\in\Z $, and the
  Fourier coefficients of $\yref(t)$ are such that  $\abs{\yrefk}=O(\abs{k}^{-3})$.

  We can construct the controller as in Definition~\ref{def:contrinfdim} by choosing $Z=\lp[2](\C)$,  $\Ac  = \diag(i\gw_k)_{k\in\SIndset}$, $\Bc = c((1+\abs{k})^{-1/2-\eps})_{k\in\Z}$ for some small $\eps>0$, $\Cc=\Bc^\ast$, and $\Dcone=0$. The parameters $\eps>0$, $\Dc=\Dctwo>0$ and $c>0$ affect the stability properties of the closed-loop system. 
  Proposition~\ref{prop:CLexpnecess} shows that since $P(\gw_k)\to 0$ as $\abs{k}\to \infty$, the closed-loop system can not be stabilized exponentially.
  However, by  Theorem~\ref{thm:CLstabnonuniform} the closed-loop system consisting of~\eqref{eq:plantfull} and the controller with the above choices of parameters is polynomially stable.
  Indeed, since $\re P_S(i\gw) = O(\abs{\gw}^{-1/2})$ and $\abs{\Bck\inv} = (1+\abs{k})^{1/2+\eps} =O(\abs{\gw_k}^{1/2+\eps})$, we have from Theorem~\ref{thm:ContrConstNonunif}
  that $\norm{R(i\gw,A_e)}=O(\abs{\gw}^{3/2+2\eps})$ and there exist $M_e,t_0>0$ such that 
  \eq{
    \norm{T_e(t)x_{e0}} &\leq \frac{M_e}{t^{1/\ga}}\norm{A_ex_{e0}}, \qquad x_{e0} \in \Dom(A_e), \ t\geq t_0.
  }
  where $\ga = 3/2+2\eps$.

  To verify that the controller is capable of regulating the given signals $\yref(t)$ and $\wdist(t)$, we need to show that the conditions~\eqref{eq:RORPEFcondShapen} are satisfied.
  The norms $\norm{R(i\gw,A)B}$ and $\norm{R(i\gw,A)B_d}$ are uniformly bounded for large $\abs{\gw}$.
  Lemma~\ref{lem:Pitransalt} and $(\Bck^\ast)_k\in\lp[2](\C)$ imply that it is sufficient to show
  \eq{
    ( \abs{\Bck}\inv \abs{P_S(i\gw_k)}\inv(\abs{\yrefk}+\abs{P_d(i\gw_k)} \abs{\wdistk}))_{k\in\Z} \in \lp[2](\C).
  }
  The eigenfunction expansion of $A$ can be used to show $\abs{P_d(i\gw)}=O(\abs{\gw}\inv)$, and
  since $\abs{P(i\gw)\inv}=O(\abs{\gw}^{1/2})$, the above condition is satisfied for all $\yref(t)$ and $\wdist(t)$ with
  \eq{
    ( \abs{k}^{1+\eps} \abs{\yrefk})_{k\in\Z} \in \lp[2](\C) \quad \mbox{and}\quad
(\abs{k}^{\eps} \abs{\wdistk})_{k\in\Z} \in \lp[2](\C).
  }
  The condition on $(\yrefk)_k$ in particular holds for $\yref(t)$ in
  Figure~\ref{fig:heat2DexSim}.

  Finally, 
  we can study the rational rates of decay of $\norm{e(t)}$ using Theorem~\ref{thm:ContrConstNonunif}. 
  The conditions 
  in~\eqref{eq:RORPnonunifEFcond} are both satisfied if
  \eq{
    ( \abs{k}^{2+\eps} \abs{\yrefk})_{k\in\Z} \in \lp[2](\C) \quad \mbox{and}\quad
(\abs{k}^{1+\eps} \abs{\wdistk})_{k\in\Z} \in \lp[2](\C).
  }
  The first condition is satisfied for $\yref(t)$ in 
  Figure~\ref{fig:heat2DexSim}
  whenever $0<\eps<1/2$.
  Then
  for all $x_{e0}\in X_e$ such that $A_ex_{e0}+B_ev_0\in X_e$ we have
  \eqn{
    \label{eq:heatexerrrate}
    \int_t^{t+1}\norm{e(s)} ds &\leq \frac{M_e^e}{t^{1/\ga}}
    \left( \norm{A_ex_{e0}+B_e\distref(0)} 
    + \Piseqnorm \right), \qquad t\geq t_0
  }
  where $\ga = 3/2+2\eps$, and a direct estimates shows that for any fixed $\eps>0$ 
  \eq{
   \Piseqnorm 
   \lesssim \norm{(\abs{k}^{2+\eps}\abs{\yrefk}+\abs{k}^{1+\eps}\abs{\wdistk})}_{\lp[2]}.
  }
  For disturbance signals satisfying $\wdist(0)=0$, Lemma~\ref{lem:ContrAeBedomSuff} shows that~\eqref{eq:heatexerrrate} 
  holds whenever $x_0\in \Dom(A)$ and $z_0\in \Dom(\Ac )$ are such that
    $\Cc z_0  = \Dc(\CL x_0-\yref(0))$.
  Moreover, by Proposition~\ref{prop:ContrInfPointwise} the regulation error satisfies $\norm{e(t)}\to 0$ as $t\to \infty$ for all such initial states.

  For simulations 
  the solution of the controlled heat equation~\eqref{eq:heatex} was approximated with Finite Differences using a $N\times N$ grid with $N=20$. 
The free parameters of the controller were chosen as $\eps=1/10$, $c=8$, and $\Dc=15$.
  The state of the controller was approximated by truncating the infinite matrix $\Ac $ to a $31\times 31$ diagonal matrix with eigenvalues $\set{i \pi k}_{\abs{k}\leq N_S}$ for $N_S=15$.
  Figure~\ref{fig:heat2DexSim} depicts the output of the controlled heat equation for $2\leq t\leq 8$ and the behaviour of the error integrals for $0\leq t\leq 10$
for the initial state $x_0(\xi_1,\xi_2) =-(1+\xi_1^2/4-\xi_1^3/6)(\cos(\pi\xi_2)/10+2)$ such that $x_0\in \Dom(A)$ and an initial state $z_0\in \Dom(\Ac)$ satisfying $\Cc z_0 = \Dc(Cx_0-\yref(0))$.

    \begin{figure}[h!]
  \begin{minipage}{.48\linewidth}
      \begin{center}
	\includegraphics[width=1.1\linewidth]{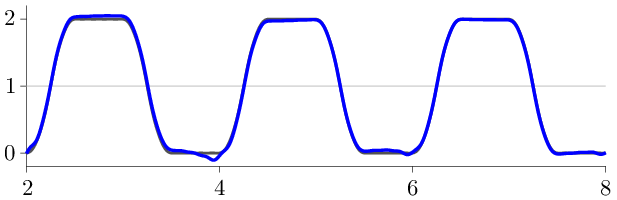}
      \end{center}
  \end{minipage}
  \hfill
  \begin{minipage}{.51\linewidth}
      \begin{flushright}
	\includegraphics[width=.9\linewidth]{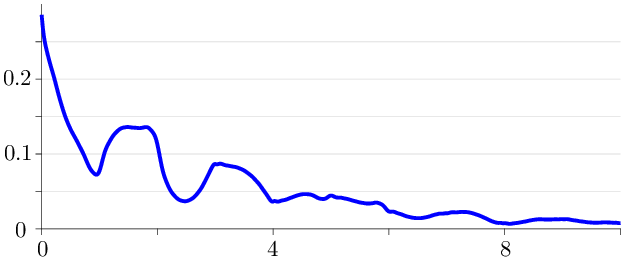}
      \end{flushright}
  \end{minipage}
      \caption{The reference $\yref(t)$ (left, gray), the output $y(t)$ (left, blue), and $\int_t^{t+1}\norm{e(s)}ds$ (right) for the heat equation.}
\label{fig:heat2DexSim}
\end{figure}

\section*{Acknowledgement}
  The author is grateful to Reinhard Stahn for discussions regarding Theorem~\ref{thm:CLstabnonuniform} and to Professor Charles Batty for helpful comments on non-uniform stability of semigroups.

\appendix

\section{}
\label{sec:OpEstimates}

\begin{lemma}
  \label{lem:Repostoinv}
  Let $X$ be a Hilbert space and let $T,S\in \Lin(X)$ be such that $\re T\geq c\geq 0$ and $\re S\geq d\geq 0$.
  \begin{itemize}
    \item[\textup{(a)}]
      If $T$ is boundedly invertible,
      then $\re T\inv \geq c\norm{T}^{-2}$.
      If $c>0$, then $T\inv $ exists and $\norm{T\inv}\leq \frac{1}{c}$.

    \item[\textup{(b)}]  If
      $c>0$ or $d>0$, then 
      $\norm{T(I+ST)\inv}\leq \frac{\norm{T}^2}{c+d \norm{T}^2}$. If $c>0$ and $d\geq 0$, then
      \eq{
	\re T(I +ST)\inv
	\geq \frac{c^3+c^2d \norm{T}^2}{\norm{T}^2(1+c\norm{S})^2}  .
      }
    \item[\textup{(c)}]
      If $T$ is invertible, $c\geq 0$, and $d>0$,
      then $\re T(I+ST)\inv \geq d (\norm{T\inv}+\norm{S})^{-2}$.
    \item[\textup{(d)}] If $c\geq 0$ and $S\geq 0$, then $I+ST$ and $I+TS$ are boundely invertible, and $\re T(I+ST)\inv\geq 0$.
\end{itemize}
\end{lemma}

\begin{proof}
  (a): The proof of the first part is elementary and latter claims follow from the estimate
  \ieq{
    \norm{Tx}\norm{x}\geq \abs{\iprod{Tx}{x}}\geq \re \iprod{Tx}{x}\geq c \norm{x}^2
  }
  for $x\in X$.

  (b): If $c>0$, we can use part~(a) and $T(I+ST)\inv = (T\inv +S)\inv$.
  If $d>0$, then an argument similar to the one used in~\cite[Lem. 2.3]{CurWei06} shows that $\norm{T(I+ST)\inv}\leq \frac{1}{d}$.

  (c): The claim follows from $T(I+ST)\inv = (T\inv + S)\inv$ and part~(a).

  (d):
  Assume $\re T\geq 0$ and $S\geq 0$. 
The invertibility of $I+ST$ implies that also $I+TS$ is invertible.
It is straightforward to show that the range of $I+ST$ is dense in $X$.
Thus it suffices to show that $I+ST$ is lower bounded.
If this is not true there exists a sequence
 $(x_n)_n\subset X$ such that $\norm{x_n}=1$ for all $n\in\N$ and $\norm{(I+ST)x_n}\to 0$ as $n\to \infty$.
Then 
\ieq{
0\from \re \iprod{(I+ST)x_n}{Tx_n} 
\geq   \norm{S^{1/2} Tx_n}^2,
}
and further $\norm{STx_n}\to 0$ as $n\to \infty$.
However, since $\norm{x_n}=1$, we would then have $\norm{(I+ST)x_n}\not\to 0
$ as $n\to \infty$,  which is a contradiction.
Finally, the proof of $\re T(I+ST)\inv\geq 0$ is elementary.
\end{proof}

\begin{lemma}
  \label{lem:IDPinvert}
  Let 
  $P(\cdot):\overline{\C_+} \to \Lin(Y)$
  be such that  $\re P(\gl)\geq 0$ for all $\gl\in \overline{\C_+}$ and let $\Dc \geq 0$. Then $-1\in\rho(\Dc P(\gl))$ for all $\gl\in \overline{\C_+}$.
  If $\sup_{\gl\in \overline{\C_+}}\norm{P(\gl)}<\infty$,
  then in addition
  $\sup_{\gl\in \overline{\C_+}}\norm{(I+\Dc P(\gl))\inv}<\infty$.
\end{lemma}

\begin{proof}
   The property that $-1\in \rho(\Dc P(\gl))$ for all $\gl\in \overline{\C_+}$ follows from Lemma~\ref{lem:Repostoinv}(d). 
Assume $\sup_{\gl\in \overline{\C_+}}\norm{P(\gl)}<\infty$.
In order to show that $(I+\Dc P(\gl))\inv $ are uniformly bounded for $\gl\in \overline{\C_+}$ it is sufficient to show that there exists a constant $r>0$ such that $\norm{(I+\Dc P(\gl))u}\geq r \norm{u}$ for all $u\in U $ and $\gl \in \overline{\C_+}$. 
If no such $r>0$ exists, we can choose sequences $(\gl_n)_n\subset \overline{\C_+}$ and $(u_n)_n\subset U$ with $\norm{u_n}=1$ for all $n\in\N$ such that $\norm{(I+\Dc P(\gl_n))u_n}\to 0$ as $n\to \infty$.
Then 
\eq{
0&\from \re \iprod{(I+\Dc P(\gl_n))u_n}{P(\gl_n)u_n} 
\geq  \norm{\Dc^{1/2} P(\gl_n)u_n}^2, 
}
which implies $\norm{\Dc P(\gl_n)u_n}\to 0$ as $n\to \infty$.
However, since $\norm{u_n}=1$, we 
would then have $\norm{(I+\Dc P(\gl_n))u_n}\not \to 0$ as $n\to \infty$, which is a contradiction.
\end{proof}

  The last lemma concerns output feedback for passive systems.
  Several additional results on this topic can be found in~\cite{GuiLog17b}.

\begin{lemma}
  \label{lem:Afbreg}
    Assume $\PARsys$ is a passive regular linear system and $\gs(A)\subset \C_-$.
  If $\Dc\geq 0$, then the system $(A-B\Dc Q_1\CL,BQ_2,Q_1\CL,Q_1D)$ with $Q_1=(I+D\Dc)\inv$ and $Q_2=(I+\Dc D)\inv$ is regular, passive, and strongly stable in such a way that
  $ \gs(A-B\Dc Q_1\CL)\subset \C_-$. If $A$ generates an exponentially stable semigroup, then the same is true for $A-B\Dc Q_1\CL$.
\end{lemma}

\begin{proof}
  The system $(A-B\Dc Q_1\CL,BQ_2,Q_1\CL,Q_1D)$ is
  obtained from~\eqref{eq:plantintro} with output feedback $u(t)=-\Dc y(t)$. 
  The regularity follows from~\cite{Wei94}, since $-\Dc$ is an admissible output feedback operator by Lemma~\ref{lem:Repostoinv}(d).
  Since $ \Dc\geq 0$, it is straightforward to verify that $(A-B\Dc Q_1\CL,BQ_2,Q_1\CL,Q_1D)$ is passive.
  In particular $A-B\Dc Q_1\CL$ generates a contraction semigroup, and the strong stability of the semigroup follows from the Arendt--Batty--Lyubich--V\~{u} Theorem~\cite{AreBat88,LyuVu88} once we have shown $i\R\subset \gs(A-B\Dc Q_1\CL)$.
Let $\gl\in \overline{\C_+}$. 
  The transfer function $P(\gl)=\CL R(\gl,A)B+D$ satisfies $\re P(\gl)\geq 0$, and thus
  the operator
  \ieq{
    I+D\Dc+\CL R(\gl,A)B \Dc
    =  I+ P(\gl)\Dc
  }
  is boundedly invertible
  by Lemma~\ref{lem:Repostoinv}(d). 
Using Lemma~\ref{lem:Woodbury} we therefore see that  
 $\gl\in \rho(A-B\Dc Q_1\CL)$ and 
  \eq{
  \MoveEqLeft[5] R(\gl,A-B\Dc Q_1\CL) 
  = R(\gl,A) 
  - R(\gl,A)B (I+ \Dc P(\gl))\inv \Dc\CL R(\gl,A).
  }
  Since $\gl\in \overline{\C_+}$ was arbitrary, we have $\gs(A-B\Dc Q_1\CL)\subset \C_-$. If $A$ generates an exponentially stable semigroup, then 
$\sup_{\gl\in \overline{\C_+}}\norm{(I+\Dc P(\gl))\inv}<\infty$ by Lemma~\ref{lem:IDPinvert}, and the regularity and exponential stability of $\PARsys$ imply 
  $\sup_{\gl\in \overline{\C_+}}\norm{R(\gl,A-B\Dc Q_1\CL)}<\infty$. Thus the semigroup generated by $A-B\Dc Q_1\CL$ is exponentially stable.
\end{proof}

 \begin{proof}[Proof of Lemma~\textup{\ref{lem:Woodbury}}]
   Let $\gl\in\rho(A)$ be such that $Q\inv + \CL R(\gl,A)B$ has a bounded inverse. Denote  $R_\gl = R(\gl,A)$ and $R(\gl)=R_\gl - R_\gl B(Q\inv + \CL R_\gl B)\inv \CL R_\gl$.
   If $x\in X$, then $R(\gl)x\in X_B$ and a computation on $X_{-1}$ shows
   \eq{
     \MoveEqLeft (\gl-A+BQ\CL)R(\gl)x\\
     &= x +  B\left[Q  - (I + Q\CL R_\gl B)(Q\inv + \CL R_\gl B)\inv \right]\CL R_\gl x = x\in X.
   }
   Thus $R(\gl)x\in \Dom(A-BQ\CL)$ and $(\gl-A+BQ\CL)R(\gl)=I$.
   On the other hand, if $x\in \Dom(A-BQ\CL)$, then $x\in X_B$ and
we can again compute on $X_{-1}$ (considering $R(\gl)$ as an operator $R(\gl):X+\ran(B)\to X$)
   \eq{
     \MoveEqLeft R(\gl)(\gl-A+BQ\CL)x\\
     &= x +  R_\gl B\left[Q  - (Q\inv + \CL R_\gl B)\inv (I + \CL R_\gl BQ)\right]\CL  x = x.
   }
   Since $x\in \Dom(A-BQ\CL)$ was arbitrary, this completes the proof.
 \end{proof}

\end{document}